\documentclass[10pt]{amsart}
\usepackage{amsmath}
\usepackage{amssymb,amscd}
\usepackage{graphicx}
\usepackage{mathdots}
\usepackage{mathrsfs}       
\usepackage{graphicx}
\usepackage{bbm} 
\usepackage{microtype} 

\usepackage{colonequals} 
\usepackage{mathtools}

\usepackage{tikz,tkz-euclide}
\usepackage{tikz-cd}
\usetikzlibrary{arrows}
\usepackage{tikz-3dplot}

\usepackage{subcaption} 

\usepackage{color}\definecolor{darkblue}{rgb}{0,0.1,.5}
\usepackage[colorlinks=true,linkcolor=darkblue, urlcolor=darkblue, citecolor=darkblue]{hyperref}
\usepackage{cleveref}

\theoremstyle{plain}
\newtheorem{theorem}{Theorem}[section]

\newtheorem{proposition}[theorem]{Proposition}
\newtheorem{corollary}[theorem]{Corollary}

\theoremstyle{definition}
\newtheorem{definition}[theorem]{Definition}
\newtheorem{example}[theorem]{Example}

\theoremstyle{remark}

\numberwithin{equation}{section}

\def \begineq{\begin{equation}}
\def \endeq{\end{equation}}

\def \bb{\mathbb}

\def \RR{{\bb{R}}}
\def \ZZ{{\bb{Z}}}

\def \bideg{\mathop{\mathrm{bideg}}}

\def \Tor{\mathop{\mathrm{Tor}}\nolimits}

\def\IL{\mathit{IL}}
\def\CH{\mathit{CH}}
\def\HH{\mathit{HH}}
\def\PH{\mathcal{PH}}

\def\PHHZ{\mathcal{PHHZ}}
\def\BB{\mathit{BB}}
\def\BDB{\mathbb{BB}}
\def\dgh{\mathop{d_{\mathit{GH}}}}
\def\dpgh{\mathop{d'_{\mathit{GH}}}}
\def\HF{\mathit{HF}}

\def\zk{\mathcal Z_K}

\DeclareMathAlphabet{\mathbbmsl}{U}{bbm}{m}{sl}

\newcommand{\ts}{\textsc}

\makeatletter
\@namedef{subjclassname@2020}{%
  \textup{2020} Mathematics Subject Classification}
\makeatother

\begin{document}

\title[Stability theorem for bigraded barcodes]
{A stability theorem for bigraded persistence 
barcodes}

\author[Bahri]{Anthony Bahri}
\address{Department of Mathematics, Rider University, Lawrenceville, NJ 08648, USA}
\email{bahri@rider.edu}

\author[Limonchenko]{Ivan Limonchenko}
\address{Mathematical Institute of the Serbian Academy of Sciences and Arts (SASA),
Belgrade, Serbia}
\email{ivan.limoncenko@turing.mi.sanu.ac.rs}

\author[Panov]{Taras Panov}
\address{Department of Mathematics and Mechanics, Moscow
State University, Russia;\newline
Institute for Information Transmission Problems, Russian Academy of Sciences, Moscow, Russia;\newline
National Research University Higher School of Economics, Moscow, Russia}
\email{tpanov@mech.math.msu.su}

\author[Song]{Jongbaek Song}
\address{Department of Mathematics Education, Pusan National University, Busan, Republic of Korea}
\email{jongbaek.song@pusan.ac.kr}

\author[Stanley]{Donald Stanley}
\address{Department of Mathematics and Statistics, University of Regina, Regina, Saskatchewan, Canada}
\email{Donald.Stanley@uregina.ca}




\begin{abstract} 
We define bigraded persistent homology modules and bigraded barcodes of a finite pseudo-metric space $X$ using the ordinary and double homology of the moment-angle complex associated with the Vietoris--Rips filtration of~$X$. We prove a stability theorem for the bigraded persistent double homology modules and barcodes. 
\end{abstract}

\maketitle

\section{Introduction}\label{intro}
Given a finite pseudo-metric space $(X,d_X)$, the Vietoris--Rips filtration is a sequence $\{R(X,t)\}_{t\ge0}$ of filtered flag simplicial complexes  associated with~$X$. The simplicial homology of $R(X,t)$ is used to define the most basic persistence modules in data science, the persistent homology of~$X$.

In toric topology, a finer homological invariant of a simplicial complex $K$ is considered, the bigraded homology of the moment-angle complex $\mathcal Z_K$ associated with~$K$. The moment-angle complex $\mathcal Z_K$ is a toric space patched from products of discs and circles parametrized by simplices in a simplicial complex~$K$. It has a bigraded cell decomposition and the corresponding bigraded homology groups $H_{-i,2j}(\mathcal Z_K)$ contain the simplicial homology groups $H_k(K)$ as a direct summand. Algebraically, the bigraded homology modules $H_{-i,2j}(\mathcal Z_K)$ are the bigraded components of the $\Tor$-modules of the Stanley--Reisner ring $\mathbf k[K]$ and can be expressed via the Hochster decomposition as the sum of the reduced simplicial homology groups of all full subcomplexes $K_I$ of~$K$.

The bigraded homology of the moment-angle complexes $\mathcal Z_{R(X,t)}$ associated with the Vietoris--Rips filtration $\{R(X,t)\}_{t\ge0}$ can be used to define bigraded persistent homology modules and bigraded barcodes of a point cloud (data set)~$X$, as observed in~\cite{l-p-s-s}. Similar ideas have been pursued in~\cite{l-f-l-x23, li-xi22}. Simple examples show that bigraded persistent homology can distinguish between point clouds that are indistinguishable by the ordinary persistent homology (see Example~\ref{2sets}).

However, the approach to define bigraded persistence via the homology of moment-angle complexes has two fundamental drawbacks. First, calculating simplicial homology groups of all full subcomplexes (or samples) in $R(X,t)$ is much more resource-consuming from the computational viewpoint. Second, bigraded persistent homology lacks a stability property, which is important for applications of persistent homology in data science~\cite{c-e-h07}. 

Bigraded double homology of moment-angle complexes was introduced and studied in~\cite{l-p-s-s}. Besides theoretical interest in toric topology and polyhedral products theory, bigraded double homology fixes both drawbacks of ordinary bigraded homology mentioned above, therefore opening a way for a more efficient application in data analysis. 

Double homology $\HH_*(\mathcal Z_K)$ is the homology of the chain complex $\CH_*(\mathcal Z_K)=(H_*(\mathcal Z_K),\partial')$ obtained by endowing the bigraded homology of~$\mathcal Z_K$ with a second differential~$\partial'$. 
The bigraded double homology modules are smaller than the bigraded homology modules, and therefore ought to be more accessible from a computational viewpoint. More importantly, persistent homology modules defined from the bigraded double homology of the Vietoris--Rips filtration have a stability property, which says roughly that the bigraded barcode is robust to changes in the input data. 

In Section~\ref{sec-prelim} we review persistence modules, standard persistent homology and barcodes. 

In Section~\ref{macbbn} we define the bigraded persistent homology and persistent double homology modules and the corresponding bigraded barcodes.

The Stability Theorem for persistent homology says that the Gromov--Hausdorff distance between two finite pseudo-metric spaces (data sets) is bounded from below by half the interleaving distance between their persistent homology modules. The interleaving distance between persistent homology modules can be identified, via the Isometry Theorem, with the $\infty$-Wasserstein (bottleneck) distance between the corresponding barcodes. 

In Section~\ref{sec_stability_HZK} we prove a Stability Theorem for bigraded persistent double homology (Theorem~\ref{thm_BB_HH_stable}). This is done in two stages. First, we prove that bigraded persistent homology satisfies a stability property with respect to a modified Gromov--Hausdorff distance, in which the infimum is taken over bijective correspondences only (Theorem~\ref{thm_BB_stable}). Second, stability for bigraded persistent double homology is established using the invariance of double homology with respect to the doubling operation on simplicial complexes.

\subsection*{Acknowledgements}
Bahri was supported by the Simons Foundation Grant (426160). 
Limonchenko was supported by the Serbian Ministry of Science, Technological Development and Innovation through the Mathematical Institute of the Serbian Academy of Sciences and Arts. Research of Panov was funded within the framework of the HSE University Basic Research Program. Song was supported by the National Research Foundation of Korea(NRF) grant funded by the Korea government(MSIT) (RS-2025-00555914). He was also supported by a KIAS Individual Grant (SP076101) at Korea Institute for Advanced Study. Stanley is supported by NSERC RGPIN-05466-2020. 
The authors are grateful to Daniela Egas Santander for helpful conversations at the Fields Institute. This work began at the Fields Institute during the Thematic Program on Toric Topology and Polyhedral Products.

\section{Preliminaries}\label{sec-prelim}

\subsection{Persistence modules and barcodes}\label{sec_persistent_module}

Consider the set $\RR_{\geq 0}$ of nonnegative real numbers, which we regard as a poset category with respect to the standard inequality $\leq$.  A \emph{persistence module} is a (covariant) functor 
\begin{equation}\label{eq_persistent_module}
    \mathcal{M} \colon \mathbb{R}_{\geq 0} \to \mathbf{k}\ts{-mod}
\end{equation}
where $\mathbf{k}\ts{-mod}$ denotes the category of modules over a principal ideal domain $\mathbf{k}$. A persistence module can be thought of as a family of $\mathbf{k}$-modules $\{M_s\}_{s\in \RR_{\geq 0}}$ together with morphisms $\{\phi_{s_1,s_2} \colon M_{s_1} \to M_{s_2} \}_{s_1\leq  s_2}$ such that $\phi_{s,s}$ is the identity on $M_s$ and $\phi_{s_2,s_3}\circ \phi_{s_1,s_2} = \phi_{s_1,s_3}$ whenever $s_1\leq  s_2\leq  s_3$ in $\RR_{\geq 0}$. 

In a more general version of persistence modules, the domain category $\mathbb{R}_{\geq 0}$ can be replaced by any small category and the codomain $\mathbf{k}\ts{-mod}$ by any Grothendieck category. See for example \cite{b-s-s}. 

\begin{example}\label{ex_interval_module}
An \emph{interval} is a subset $I\subset\mathbb R_{\geq0}$ such that if $r\in I$ and $t\in I$ and $r<s<t$, then $s\in I$ (such an interval may be closed or open from each side, and may have infinite length).
Given an interval $I$, we define 
\[
\mathbf{k}^I_s
\colonequals \begin{cases} \mathbf{k} & \text{if } s\in I;\\
  \mathbf{0} & \text{otherwise}.
  \end{cases}
\]
Then, the family $\{ \mathbf{k}^I_s \}_{s\in \RR_{\geq 0}}$ together with the family $\{\phi_{s_1, s_2}\colon \mathbf{k}^I_{s_1} \to \mathbf{k}^I_{s_2}\}_{s_1 \leq s_2}$ of morphisms, where $\phi_{s_1, s_2}$ is the identity if $s_1, s_2\in I$ and the zero map otherwise, forms a persistence module. We denote it by $\mathbf{k}(I)$  and refer to it as the \emph{interval module} corresponding to $I$. 
\end{example}

The direct sum $\mathcal{M}\oplus \mathcal{N}$ of two persistence modules $\mathcal{M}=\{M_s\}_{s\in \RR_{\geq 0}}$ and $\mathcal{N}=\{N_s\}_{s\in \RR_{\geq 0}}$ is defined pointwise as the family of $\mathbf{k}$-modules $\{M_s\oplus N_s\}_{s\in \RR_{\geq 0}}$. The following decomposition theorem for a persistence module is originally due to Zomorodian and Carlsson~\cite[Theorem~2.1]{CZ05} for the case where the domain category is $\ZZ_{\geq 0}$ and $\mathbf{k}$ is a field. It is generalized to the case where the domain category is~$\RR_{\geq 0}$ or $\mathbb{R}$ by Crawley-Boevey \cite[Theorem 1.1]{Cr}. We also refer to \cite[Theorem~1.2]{BoCr} for a simpler proof. 

\begin{theorem}\label{prop_decomp_persitent_module}
Let $\mathcal{M}=\{M_s\}_{s\in \RR_{\geq 0}}$ be a persistence module as in \eqref{eq_persistent_module}. If  $\mathbf{k}$ is a field and all $M_s$ are finite dimensional $\mathbf{k}$-vector spaces, then 
\begin{equation}\label{eq_decomp_persistent_module}
\mathcal{M}=\bigoplus_{I\in B(\mathcal{M})} \mathbf{k}(I)
\end{equation}
for some multiset $B(\mathcal{M})$ of intervals in $\RR_{\geq 0}$, where $\mathbf{k}(I)$ is the interval module defined in Example \ref{ex_interval_module}.
\end{theorem}
The above result gives us a discrete invariant of a persistence module $\mathcal{M}$, namely the multiset $B(\mathcal{M})$  of \eqref{eq_decomp_persistent_module}, which is called  the \emph{barcode} of $\mathcal{M}$.

\subsection{The Vietoris--Rips filtration and persistent homology}\label{sec_persistent_homology}
A typical example that fits into Theorem~\ref{prop_decomp_persitent_module} is the persistent homology of filtered simplicial complexes. 

A \emph{finite pseudo-metric space} is a nonempty finite set $X$ together with a function 
\[
d_X\colon X\times X \to \RR_{\geq 0}
\]
satisfying  $d_X(x,x)=0$, $d_X(x,y)=d_X(y,x)$ and $d_X(x,z)\leq d_X(x,y)+d_X(y,z)$ for every $x,y,z\in X$. Note that two discrete points may have distance $0$ in~$X$. We can think of a finite pseudo-metric space $(X, d_X)$ as a point cloud. 

The \emph{Vietoris--Rips filtration} $\{R(X,t)\}_{t\geq 0}$ associated with $(X,d_X)$ consists of the Vietoris--Rips simplicial complexes~$R(X,t)$. The latter is defined as the clique complex of the graph whose vertex set is $X$ and two vertices $x$ and $y$ are connected by an edge if $d_X(x,y)\leq t$. We have a simplicial inclusion 
\begin{equation}\label{eq_inclusion_Rips_complex}
R(X,t_1) \hookrightarrow  R(X, t_2) 
\end{equation}
whenever $t_1 \leq t_2$. See Figure \ref{fig_barcode_example} for an example of Vietoris--Rips complexes.   
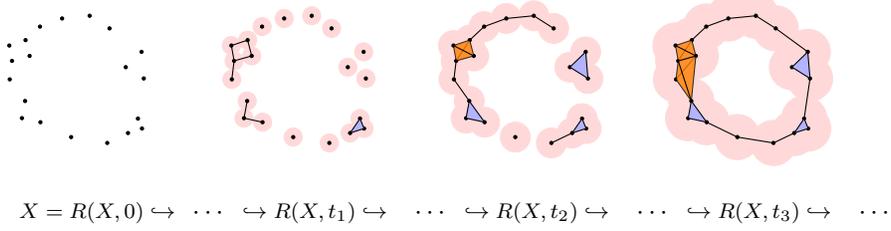
\begin{figure}
\begin{tikzpicture}[scale=0.6]
\coordinate (v1) at (9.74,7.01); 
\coordinate (v2) at (9.33,7.10);

\coordinate (v19) at (9.4,7.48);

\coordinate (v3) at (9.06,7.96);
\coordinate (v4) at (9.11,8.37);

\coordinate (v18) at (9.5,8.48);
\coordinate (v20) at (9.05,8.71);

\coordinate (v5) at (9.41,8.83);
\coordinate (v6) at (9.73,9.13);
\coordinate (v7) at (10.22,9.31);
\coordinate (v8) at (10.83,9.37);
\coordinate (v9) at (11.27,9.09);
\coordinate (v10) at (11.98,8.57);
\coordinate (v11) at (11.63,8.23);
\coordinate (v12) at (12.03,7.98);
\coordinate (v13) at (11.90,7.09);
\coordinate (v14) at (11.99,6.87);
\coordinate (v15) at (11.68,6.77);
\coordinate (v16) at (11.22,6.55);
\coordinate (v17) at (10.42,6.68);

\node at (11, 5) {\footnotesize$X=R(X,0)  \hookrightarrow  $};
\node at (13.5, 5) {$\cdots$};

\foreach \x in {1,2,...,20} {
\filldraw (v\x) circle (1pt);
}

\begin{scope}[xshift=140]
\coordinate (v1) at (9.74,7.01); 
\coordinate (v2) at (9.33,7.10);

\coordinate (v19) at (9.4,7.48);

\coordinate (v3) at (9.06,7.96);
\coordinate (v4) at (9.11,8.37);

\coordinate (v18) at (9.5,8.48);
\coordinate (v20) at (9.05,8.71);

\coordinate (v5) at (9.41,8.83);
\coordinate (v6) at (9.73,9.13);
\coordinate (v7) at (10.22,9.31);
\coordinate (v8) at (10.83,9.37);
\coordinate (v9) at (11.27,9.09);
\coordinate (v10) at (11.98,8.57);
\coordinate (v11) at (11.63,8.23);
\coordinate (v12) at (12.03,7.98);
\coordinate (v13) at (11.90,7.09);
\coordinate (v14) at (11.99,6.87);
\coordinate (v15) at (11.68,6.77);
\coordinate (v16) at (11.22,6.55);
\coordinate (v17) at (10.42,6.68);

\node at (10.9, 5) {\footnotesize$\hookrightarrow  R(X,t_1)  \hookrightarrow$};
\node at (13.5, 5) {$\cdots$};

\foreach \x in {1,2,...,20} {
\draw[fill=red!15, red!15] (v\x) circle (0.2);
}

\foreach \x in {1,2,...,20} {
\filldraw (v\x) circle (1pt);
}

\draw (v1)--(v2)--(v19);
\draw (v3)--(v4)--(v18)--(v5)--(v20)--(v4);
\draw[fill=blue!30] (v13)--(v14)--(v15)--cycle;
\end{scope}

\begin{scope}[xshift=280]
\coordinate (v1) at (9.74,7.01); 
\coordinate (v2) at (9.33,7.10);

\coordinate (v19) at (9.4,7.48);

\coordinate (v3) at (9.06,7.96);
\coordinate (v4) at (9.11,8.37);

\coordinate (v18) at (9.5,8.48);
\coordinate (v20) at (9.05,8.71);

\coordinate (v5) at (9.41,8.83);
\coordinate (v6) at (9.73,9.13);
\coordinate (v7) at (10.22,9.31);
\coordinate (v8) at (10.83,9.37);
\coordinate (v9) at (11.27,9.09);
\coordinate (v10) at (11.98,8.57);
\coordinate (v11) at (11.63,8.23);
\coordinate (v12) at (12.03,7.98);
\coordinate (v13) at (11.90,7.09);
\coordinate (v14) at (11.99,6.87);
\coordinate (v15) at (11.68,6.77);
\coordinate (v16) at (11.22,6.55);
\coordinate (v17) at (10.42,6.68);

\node at (10.9, 5) {\footnotesize$\hookrightarrow  R(X,t_2)  \hookrightarrow$};
\node at (13.5, 5) {$\cdots$};

\foreach \x in {1,2,...,20} {
\draw[red!15, fill=red!15] (v\x) circle (0.33);
}

\foreach \x in {1,2,...,20} {
\filldraw (v\x) circle (1pt);
}

\draw[fill=blue!30] (v1)--(v2)--(v19)--cycle;
\draw (v19)--(v3)--(v4);
\draw (v4)--(v5);
\draw[fill=orange, opacity=0.8] (v4)--(v18)--(v5)--(v20)--cycle;
\draw (v20)--(v18);
\draw (v5)--(v6)--(v7)--(v8)--(v9);
\draw[fill=blue!30] (v10)--(v11)--(v12)--cycle;

\draw[fill=blue!30] (v13)--(v14)--(v15)--cycle;
\draw (v15)--(v16);
\end{scope}

\begin{scope}[xshift=420]
\coordinate (v1) at (9.74,7.01); 
\coordinate (v2) at (9.33,7.10);

\coordinate (v19) at (9.4,7.48);

\coordinate (v3) at (9.06,7.96);
\coordinate (v4) at (9.11,8.37);

\coordinate (v18) at (9.5,8.48);
\coordinate (v20) at (9.05,8.71);

\coordinate (v5) at (9.41,8.83);
\coordinate (v6) at (9.73,9.13);
\coordinate (v7) at (10.22,9.31);
\coordinate (v8) at (10.83,9.37);
\coordinate (v9) at (11.27,9.09);
\coordinate (v10) at (11.98,8.57);
\coordinate (v11) at (11.63,8.23);
\coordinate (v12) at (12.03,7.98);
\coordinate (v13) at (11.90,7.09);
\coordinate (v14) at (11.99,6.87);
\coordinate (v15) at (11.68,6.77);
\coordinate (v16) at (11.22,6.55);
\coordinate (v17) at (10.42,6.68);

\node at (10.9, 5) {\footnotesize$\hookrightarrow  R(X,t_3)  \hookrightarrow$};
\node at (13.5, 5) {$\cdots$};

\foreach \x in {1,2,...,20} {
\draw[red!15, fill=red!15] (v\x) circle (0.5);
}

\foreach \x in {1,2,...,20} {
\filldraw (v\x) circle (1pt);
}

\draw[fill=blue!30] (v1)--(v2)--(v19)--cycle;
\draw (v3)--(v18);
\draw[fill=orange, opacity=0.8] (v19)--(v3)--(v4)--(v18)--cycle;
\draw (v19)--(v4);

\draw (v4)--(v5);
\draw[fill=orange, opacity=0.8] (v4)--(v18)--(v5)--(v20)--cycle;
\draw (v20)--(v18);
\draw (v5)--(v6)--(v7)--(v8)--(v9)--(v10);
\draw[fill=blue!30] (v10)--(v11)--(v12)--cycle;

\draw[fill=blue!30] (v13)--(v14)--(v15)--cycle;
\draw (v15)--(v16)--(v17)--(v1);
\draw (v12)--(v13);
\end{scope}

\end{tikzpicture}
    \caption{A point cloud and the corresponding Vietoris--Rips filtration.}
    \label{fig_barcode_example}
\end{figure}

The $n$-dimensional \emph{persistent homology} module 
\[
  \mathcal{PH}_n(X) \colon \RR_{\geq 0} \to \mathbf{k}\ts{-mod},\quad
  t\mapsto \widetilde{H}_n(R(X,t)),
\]
maps $t\in \RR_{\geq0}$ to the reduced simplicial homology group $\widetilde{H}_n(R(X,t))$ with coefficients in $\mathbf k$ and maps a morphism $t_1\leq t_2$ to the homomorphism $\widetilde{H}_n(R(X, t_1))\to \widetilde{H}_n(R(X, t_2))$ induced by~\eqref{eq_inclusion_Rips_complex}. We also define the graded persistent homology module
\begin{equation}\label{eq_PH}
 \mathcal{PH}(X)=\bigoplus_{n\ge0}\mathcal{PH}_n(X).
\end{equation}

Applying Theorem~\ref{prop_decomp_persitent_module} to the graded persistence module~\eqref{eq_PH} for a field $\mathbf{k}$, we obtain the barcode of $\mathcal{PH}(X)$, which we denote simply by~$B(X)$. In this case, $B(X)$ can be described more explicitly as follows.

A homology class $\alpha\in\widetilde H_n(R(X,t))$ is said to
\begin{itemize}
\item[(1)] \emph{be born} at $r$ if
\begin{itemize}
    \item[(i)] $\alpha\in\mathop\mathrm{im}\bigl(\widetilde H_n(R(X,r))\to \widetilde H_n(R(X,t))\bigr)$;
    \item[(ii)] $\alpha\notin\mathop\mathrm{im}\bigl(\widetilde H_n(R(X,p))\to \widetilde H_n(R(X,t))\bigr)$ for $p<r$,
\end{itemize}
\item[(2)] \emph{die} at $s$ if 
\begin{itemize}
    \item[(i)] $\alpha\in\ker\bigl(\widetilde H_n(R(X,t))\to \widetilde H_n(R(X,s))\bigr)$;
    \item[(ii)] $\alpha\notin\ker\bigl(\widetilde H_n(R(X,t))\to \widetilde H_n(R(X,q))\bigr)$ for $q<s$.
\end{itemize} 
We set $s=\infty$ if $\alpha\notin\ker\bigl(\widetilde H_n(R(X,t))\to \widetilde H_n(R(X,q))\bigr)$ for any~$q$.
\end{itemize}

If $\alpha\in\widetilde H_n(R(X,t))$ is born at $r$ and dies at $s$, then $[r,s)$ is called the \emph{persistence interval} of~$\alpha$. Then, Theorem \ref{prop_decomp_persitent_module} says that there exists a choice of generators of the homology groups $\widetilde{H}_\ast(R(X,t))$ whose persistence intervals are consistent with the barcode $B(X)$.
For each $t\in\mathbb R_{\geq 0}
$, the dimension of
$\widetilde H_n(R(X,t))$ is equal to the number of $n$-dimensional intervals in the barcode containing~$t$.

\begin{figure}
  \centering
    \includegraphics[width=0.9\linewidth]{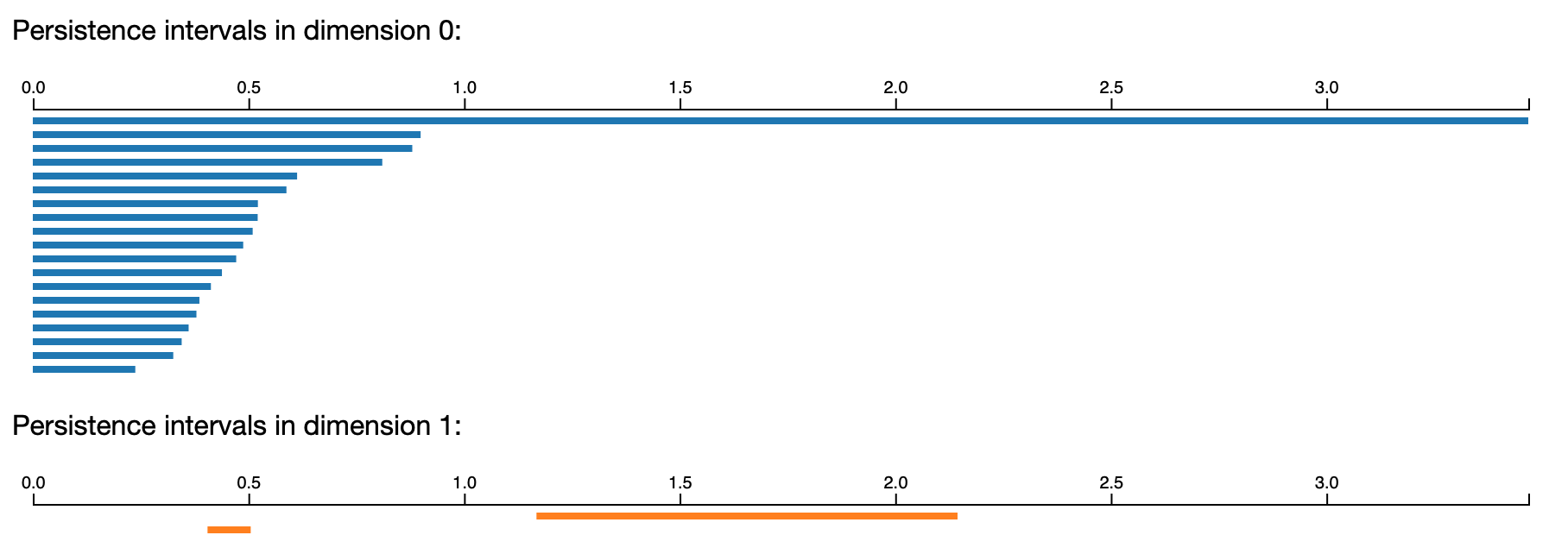}
    \caption{The barcode corresponding to the Vietoris--Rips complex in Figure \ref{fig_barcode_example}. The image is generated by Ripser for Python~\cite{Ripser}.}
    \label{fig_barcode}
\end{figure}

Since $(X, d_X)$ is a finite pseudo-metric space, it follows that $R(X,t)$ is contractible when $t$ is large enough.  
Therefore, the corresponding barcode $B(X)$ is a finite collection of half-open persistence intervals with finite lengths. If we use the unreduced simplicial homology in the definition of persistent homology, then the barcode would contain a single infinite interval $[0,\infty)$ in dimension~$0$. Figure~\ref{fig_barcode} displays the unreduced barcode corresponding to the Vietoris--Rips filtration in Figure~\ref{fig_barcode_example}.

\section{Bigraded persistence and double 
cohomology}\label{macbbn}

\subsection{Cohomology of moment-angle complexes}\label{subsec_HZK}
Let $K$ be a simplicial complex on the set $[m]=\{1,2,\ldots,m\}$. We refer to a subset $I=\{ i_1,\ldots,i_k \} \subset[m]$ that is contained in $K$ as a \emph{simplex}. A one-element simplex $\{i\}\in K$ is a \emph{vertex}. We also assume that $\varnothing\in K$ and, unless explicitly stated otherwise, that $K$ contains all one-element subsets $\{i\}\in[m]$ (that is, $K$ is a simplicial complex without \emph{ghost vertices}). 

Let $D^2$ be the unit disc in $\mathbb R^2$, and let $S^1$ be its boundary circle.
For each simplex $I\in K$, consider the topological space  
\begin{equation*}
  (D^2, S^1)^I \colonequals \{(z_1, \dots, z_m)\in (D^2)^m \colon |z_j|=1 \text{ if } j\notin I\}\subset (D^2)^m.
\end{equation*}
Note that $(D^2, S^1)^I$ is a natural subspace of  $(D^2, S^1)^J$ whenever $I\subset J$. The \emph{moment-angle complex} corresponding to $K$ is 
\[
  \zk\colonequals\bigcup_{I\in K}(D^2,S^1)^I\subset(D^2)^m. 
\]
We refer to \cite[Chapter~4]{BP-book} for more details and examples. 

Let $\mathbf k$ be a coefficient ring, which we assume to be a principal ideal domain. The \emph{face ring} (the \emph{Stanley--Reisner ring}) of a simplicial complex $K$ is
\[
  \mathbf k[K]\colonequals \mathbf k[v_1, \dots, v_m]/\mathcal{I}_K, 
\]
where $\mathcal{I}_K$ is the ideal generated by square-free monomials $\prod_{i\in I} v_i$ for which $I\subset[m]$ is not a simplex of~$K$. 

The following theorem summarizes several presentations of the cohomology ring~$H^\ast(\zk)$.

\begin{theorem}[{\cite[Section 4.5]{BP-book}}]\label{BPtheorem}
There
are isomorphisms of bigraded commutative algebras
\begin{align}
  H^*(\zk)&\cong\Tor_{\mathbf k[v_1,\ldots,v_m]}\bigl(\mathbf k,\mathbf k[K]\bigr) \nonumber  \\
  &\cong H\bigl(\Lambda[u_1,\ldots,u_m]\otimes
  \mathbf k[K],d\bigr) \label{eq_Koszul_cohom}\\
  &\cong \bigoplus_{I\subset[m]}\widetilde H^*(K_I). \label{eq_Hochster_decomp}
\end{align}
Here, \eqref{eq_Koszul_cohom} is cohomology of the
bigraded algebra with bidegrees $\bideg u_i=(-1,2)$, $\bideg v_i=(0,2)$ and differential $d=\sum_{i=1}^mv_i\frac{\partial}{\partial u_i}$ of bidegree $(1,0)$ (the Koszul complex). In \eqref{eq_Hochster_decomp}, each component $\widetilde H^*(K_I)$ denotes the reduced simplicial cohomology
of the full subcomplex $K_I\subset K$ (the restriction of $K$
to $I\subset[m]$). The last isomorphism is the sum of isomorphisms
\[
  H^p(\zk)\cong
  \sum_{I\subset[m]}\widetilde H^{p-|I|-1}(K_I),
\]
and the ring structure is given by the maps
\begin{equation}\label{eq_Hoschter_product}
  H^{p-|I|-1}(K_I)\otimes H^{q-|J|-1}(K_J)\to
  H^{p+q-|I|-|J|-1}(K_{I\cup J})
\end{equation}
which are induced by the canonical simplicial maps $K_{I\cup
J}\to K_I\mathbin{*} K_J $  for $I\cap J=\varnothing$ and zero otherwise. 
\end{theorem}

Here, $K_I \ast K_J=\{L\sqcup M\colon L\in K_I,\, M\in K_J\}$ is the join of $K_I$ and $K_J$. The map~\eqref{eq_Hoschter_product} can be described via the homotopy equivalence $K_I \ast K_J \simeq \Sigma (K_I \wedge K_J)$ together with the K\"unneth formula for the smash product.

Isomorphism~\eqref{eq_Hochster_decomp} is often referred to as the \emph{Hochster decomposition}, as it comes from Hochster's theorem describing $\Tor_{\mathbf k[v_1,\ldots,v_m]}(\mathbf k,\mathbf k[K])$ as a sum of cohomology of full subcomplexes~$K_I$.

The bigraded components of cohomology of $\zk$ are given by
\[
  H^{-i,2j}(\zk)\cong H^{-i,2j}\bigl(\Lambda[u_1,\ldots,u_m]\otimes
  \mathbf k[K],d\bigr)
  \cong\bigoplus_{J\subset[m],\, |J|=j}\widetilde H^{j-i-1}(K_J),\quad
\]
so that
\[
  H^p (\mathcal{Z}_K) = \bigoplus_{-i+2j=p}H^{-i,2j} (\mathcal{Z}_K)\cong \bigoplus_{J\subset[m]}\widetilde{H}^{p-|J|-1}(K_J),
\]
There is a similar description of bigraded homology of $\zk$:
\begin{equation*}
  H_p (\mathcal{Z}_K) = \bigoplus_{-i+2j=p}H_{-i,2j} (\mathcal{Z}_K)\cong \bigoplus_{J\subset[m]}\widetilde{H}_{p-|J|-1}(K_J).
\end{equation*}

The construction of moment-angle complex and its bigraded homology is functorial with respect to inclusion of simplicial complexes:

\begin{proposition}\label{funin}
An inclusion of subcomplex $K\subset L$ induces an inclusion $\mathcal Z_K\subset\mathcal Z_L$ and a homomorphism $H_{-i,2j}(\mathcal Z_K)\to H_{-i,2j}(\mathcal Z_L)$ of bigraded homology modules.
\end{proposition}

When $\mathbf k$ is a field, the \emph{bigraded Betti numbers} of $K$ (with coefficients in~$\mathbf k$) are defined by
\[
  \beta_{-i,2j}(K)\colonequals \dim H_{-i, 2j}(\mathcal{Z}_K)=\sum_{J\subset[m],\, |J|=j} \dim \widetilde{H}_{j-i-1}(K_J).
\]  
In particular, when $j=m$, we obtain $\beta_{-i,2m}(K)=\dim\widetilde{H}_{m-i-1}(K)$.
Moreover, $\dim \widetilde{H}^{j-i-1}(K_J)$ agrees with $\dim \widetilde{H}_{j-i-1}(K_J)$ obviously. 

Dimensional considerations imply that possible locations of nonzero bigraded Betti numbers form a trapezoid in the table, see Figure~\ref{fig_bi_Betti} where we assume that $K$ is $(n-1)$-dimensional. We refer to \cite[Section 4.5]{BP-book} for more details. 
\begin{figure}
\begin{tikzpicture}
\draw[fill=yellow] (0.5,4.5)--(0.5,5)--(2.5,5)--(2.5,4.5)--cycle;
\draw (0,0) grid[step=.5cm] (5,5);
\draw[stealth-] (-0.5,0)--(5,0);
\draw[-stealth] (5,0)--(5,5.5);
\node[right] at (5,5.5) {\footnotesize$2j$};
\node[left] at (-0.5,0) {\footnotesize$-i$};
\node[below] at (4.75,0) {\scriptsize$0$};
\node[below] at (4.25,0) {\scriptsize$-1$};
\node[right] at (5,0.25) {\scriptsize$0$};
\node[right] at (5,0.75) {\scriptsize$2$};
\node[right] at (5,1.25) {\scriptsize$4$};

\node at (4.75, 0.25) {\scriptsize$1$};
\foreach \x in {1,2,...,8}
	{\node at (4.75-.5*\x,0.75+.5*\x) {\scriptsize$\ast$};}
\foreach \x in {1,2,...,7}
	{\node at (4.75-.5*\x,1.25+.5*\x) {\scriptsize$\ast$};}
\foreach \x in {1,2,...,6}
	{\node at (4.75-.5*\x,1.75+.5*\x) {\scriptsize$\ast$};}
\foreach \x in {1,2,...,5}
	{\node at (4.75-.5*\x,2.25+.5*\x) {\scriptsize$\ast$};}

\foreach \x in {1,2,...,8}
	{\draw[red, very thick] (.5*\x, 5.5-.5*\x)--(.5*\x, 5-.5*\x)--(.5*\x+0.5, 5-.5*\x)  ;}
\foreach \x in {1,2,3,4}
	{\draw[red, very thick] (2+0.5*\x, 5.5-.5*\x)--(2+0.5*\x, 5-.5*\x)--(2.5+.5*\x, 5-.5*\x)  ;}
\draw[red, very thick] (0.5,5)--(2.5,5);
\draw[red, very thick] (4.5,1)--(4.5,3);

\node[right] at (5, 4.75) {\scriptsize$2m$};

\draw[<-] (0.75, 4.85)--(0.5, 5.25); \node[above] at (0.5, 5.25) {\scriptsize$\widetilde{H}_0(K)$};
\draw[<-] (2.25, 4.85)--(2.5, 5.25); \node[above] at (2.5, 5.25) {\scriptsize$\widetilde{H}_{n-1}(K)$};
\node[above] at (1.5, 5.25) {\scriptsize$\cdots$};

\draw[<-] (0.75, -0.1)--(0.75, -0.5);
\node[below] at (0.75, -0.5) {\scriptsize$-(m-1)$};
\draw[<-] (2.25, -0.1)--(2.25, -0.5);
\node[below] at (2.25, -0.5) {\scriptsize$-(m-n)$};
\end{tikzpicture}
\caption{Table of bigraded Betti numbers.}
\label{fig_bi_Betti}
\end{figure}
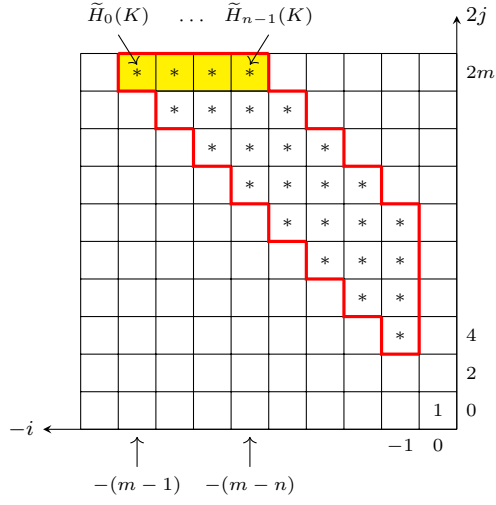

\subsection{Bigraded persistence and barcodes}\label{subsec_Bigraded_barcode}

\begin{definition}\label{bgph}
Let $(X,d_X)$ be a finite pseudo-metric space and $\{R(X,t)\}_{t\geq 0}$ its associated Vietoris--Rips filtration.
We define the \emph{persistent homology} module of bidegree $(-i,2j)$ as 
\[
  \mathcal{PHZ}_{-i,2j}(X) \colon \RR_{\geq 0} \to \mathbf{k}\ts{-mod},\quad
  t\mapsto H_{-i,2j}(\mathcal Z_{R(X,t)}).
\]
It maps $t\in \RR_{\geq0}$ to the bigraded homology module  $H_{-i,2j}(\mathcal Z_{R(X,t)})$ and maps a morphism $t_1\leq t_2$ to the homomorphism $H_{-i,2j}(\mathcal Z_{R(X,t_1)})\to H_{-i,2j}(\mathcal Z_{R(X,t_2)})$, see Proposition~\ref{funin}. We also define the bigraded persistent homology module
\begin{equation}\label{eq_PHZ}
 \mathcal{PHZ}(X)=
 \bigoplus_{\substack{0\leq i\leq m-1; \\ 0\leq j \leq m}}\mathcal{PHZ}_{-i,2j}(X).
\end{equation}

For any subset $J\subset X$, the Vietoris--Rips complex $R(J,t)$ is the full subcomplex $R(X,t)_J$ of $R(X, t)$. Hence,
the Hochster decomposition \eqref{eq_Hochster_decomp} yields the decomposition of the persistence modules:
\begin{equation}\label{hoch-pers}
  \mathcal{PHZ}(X)=\bigoplus_{J \subset X} \mathcal{PH}(J).
\end{equation}

The \emph{bigraded barcode} $\BB(X)$ is the collection of persistence intervals of generators of the bigraded homology groups $H_{-i,2j}(\mathcal Z_{R(X,t)})$. We note that bigraded persistence intervals are defined in the same way as for ordinary persistent homology, see Theorem \ref{prop_decomp_persitent_module} and Subsection~\ref{sec_persistent_homology}. For each $t\in\mathbb R_{\geq 0}$, the dimension of $H_{-i,2j}(\mathcal Z_{R(X,t)})$ is equal to the number of persistence intervals in the bigraded barcode of bidegree $(-i,2j)$ containing~$t$.
\end{definition}

We note that the bigraded persistent homology module  $\mathcal{PHZ}(X)$ defined in~\eqref{eq_PHZ} is related to subsampling from data sets \cite{GM-arXiv, CFLMRW}. Studying the set of all subsamples of size $k$ gives what is called the \emph{distributed persistence} in \cite{SWB22}. Indeed, the $k$-distributed Vietoris--Rips persistence agrees with the persistence module $\bigoplus_{0\leq i \leq m-1}\mathcal{PHZ}_{-i, 2k}(X)$. Here, we emphasize that the module $\mathcal{PHZ}(X)$ has richer algebraic structures such as cohomology products summarized in Theorem~\ref{BPtheorem}. We also refer to \cite{BBC} for abundant algebraic structures other than cohomology products.

The bigraded barcode of $X$ is drawn as a diagram in $3$-dimensional space, which contains the original barcode of $X$ in its top level. See Figure~\ref{fig_bigr_barcode}. 
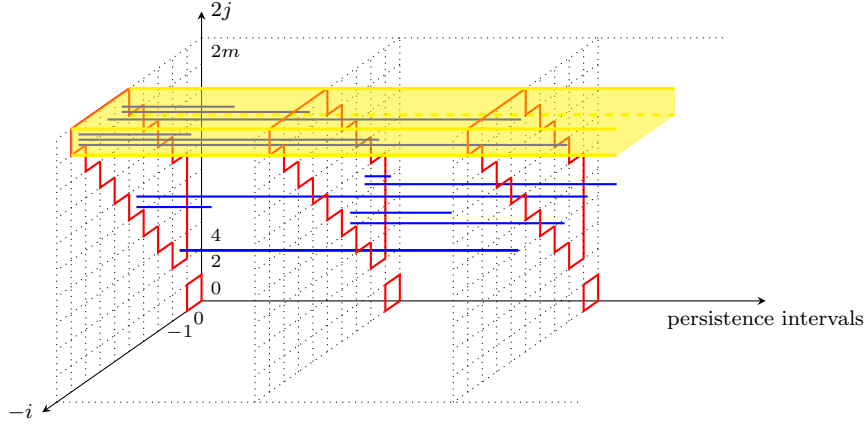
\begin{figure}
\begin{tikzpicture}[scale=0.5, yscale=0.7]
\draw[-stealth] (xyz cs:x=0) -- (xyz cs:x=15) node[below] {\footnotesize{persistence intervals}};
\draw[-stealth] (xyz cs:y=0) -- (xyz cs:y=11) node[right] {\footnotesize$2j$};
\draw[-stealth] (xyz cs:z=0) -- (xyz cs:z=11) node[left] {\footnotesize$-i$};

\node[right] at (0,9.5,0) {\scriptsize$2m$};
\node[right] at (0,.5,0) {\scriptsize$0$};
\node[right] at (0,1.5,0) {\scriptsize$2$};
\node[right] at (0,2.5,0) {\scriptsize$4$};

\node[below] at (0,0,0.2) {\scriptsize$0$};
\node[below] at (0,0,1.5) {\scriptsize$-1$};

\draw[yellow, very thick, dashed] (0,9,5)--(14.5,9,5);

\begin{scope}
\foreach \coo in {1,2,...,10}
	{\draw[dotted] (0,\coo,0)--(0,\coo,10);
	\draw[dotted] (0,0,\coo)--(0,10,\coo);
	}

\draw[red, thick] (0,0,0)--(0,1,0)--(0,1,1)--(0,0,1)--cycle;
\foreach \bd in {1,2,...,8}
	{\draw[red, thick] (0,11-\bd,10-\bd)--(0,10-\bd,10-\bd)--(0,10-\bd,9-\bd);
	}
\foreach \bd in {1,2,3,4}
	{\draw[red, thick] (0,11-\bd, 6-\bd)--(0,10-\bd, 6-\bd)--(0,10-\bd, 5-\bd);
	}
\draw[red, thick] (0,10,5)--(0,10,9);
\draw[red, thick] (0,2,1)--(0,6,1);
\end{scope}

\begin{scope}[xshift=150]
\foreach \bd in {1,2,3,4}
	{\draw[red, thick] (0,11-\bd, 6-\bd)--(0,10-\bd, 6-\bd)--(0,10-\bd, 5-\bd);
	}
\draw[red, thick] (0,2,1)--(0,6,1);
\end{scope}

\begin{scope}[xshift=300]
\foreach \bd in {1,2,3,4}
	{\draw[red, thick] (0,11-\bd, 6-\bd)--(0,10-\bd, 6-\bd)--(0,10-\bd, 5-\bd);
	}
\draw[red, thick] (0,2,1)--(0,6,1);
\end{scope}

\draw[blue, thick] (0,9.6, 8.5)--(3,9.6, 8.5);
\draw[blue, thick] (0,9.4, 8.5)--(8,9.4, 8.5);
\draw[blue, thick] (0,9.2, 8.5)--(13,9.2, 8.5);
\draw[blue, thick] (0,9.5, 5.5)--(3,9.5, 5.5);
\draw[blue, thick] (0,9.3, 5.5)--(5,9.3, 5.5);
\draw[blue, thick] (0,9.4, 6.5)--(11,9.4, 6.5);

\draw[blue, thick] (0,2.5,1.5)--(9,2.5, 1.5);
\draw[blue, thick] (0,2.5,1.5)--(9,2.5, 1.5);
\draw[blue, thick] (0,2.5,1.5)--(9,2.5, 1.5);

\draw[blue, thick] (0,5.7,4.5)--(12,5.7,4.5);
\draw[blue, thick] (0,5.3,4.5)--(2,5.3,4.5);

\draw[blue, thick] (5.3,4.7,3.5)--(8,4.7, 3.5);
\draw[blue, thick] (5.3,4.3,3.5)--(11,4.3, 3.5);

\draw[blue, thick] (5.3,5.7,2.5)--(6,5.7,2.5);
\draw[blue, thick] (5.3,5.4,2.5)--(12,5.4,2.5);

\begin{scope}[xshift=150]
\foreach \coo in {0,1,2,...,10}
	{\draw[dotted] (0,\coo,0)--(0,\coo,10);
	\draw[dotted] (0,0,\coo)--(0,10,\coo);
	}

\draw[red, thick] (0,0,0)--(0,1,0)--(0,1,1)--(0,0,1)--cycle;
\foreach \bd in {1,2,...,8}
	{\draw[red, thick] (0,11-\bd,10-\bd)--(0,10-\bd,10-\bd)--(0,10-\bd,9-\bd);
	}
\draw[red, thick] (0,10,5)--(0,10,9);
\end{scope}

\begin{scope}[xshift=300]
\foreach \coo in {0,1,2,...,10}
	{\draw[dotted] (0,\coo,0)--(0,\coo,10);
	\draw[dotted] (0,0,\coo)--(0,10,\coo);
	}

\draw[red, thick] (0,0,0)--(0,1,0)--(0,1,1)--(0,0,1)--cycle;
\foreach \bd in {1,2,...,8}
	{\draw[red, thick] (0,11-\bd,10-\bd)--(0,10-\bd,10-\bd)--(0,10-\bd,9-\bd);
	}
\draw[red, thick] (0,10,5)--(0,10,9);
\end{scope}

\draw[fill=yellow, yellow, opacity=0.5] (0,10,9)--(0,10,5)--(14.5,10,5)--(14.5,9,5)--(14.5,9,9)--(0,9,9)--cycle;
\draw[yellow, very thick] (0,10,5)--(14.5,10,5);
\draw[yellow, very thick] (0,10,9)--(14.5, 10,9);
\draw[yellow, very thick] (0,9,9)--(14.5, 9,9);

\draw[dotted] (0,0,10)--(14, 0,10);
\draw[dotted] (0,10,0)--(14, 10,0);

\end{tikzpicture}
\caption{A bigraded barcode.}
\label{fig_bigr_barcode}
\end{figure}
Here is a simple example of two sequences of point clouds with the same barcodes, but different bigraded barcodes. 

\begin{example}\label{2sets}
Let $X_1$ and $X_2$ consist of three points $(0,0), (2,0), (0,4)$ and $(0,0), (2,0), (1, \sqrt{15})$ in~$\RR^2$, respectively. The Vietoris--Rips filtration $\{R(X_1,t)\}_{t\geq 0}$ at $t=0,2,4,2\sqrt5$ is shown in Figure~\ref{fig_VR_cplxes_of_3_pts}, top, and the Vietoris--Rips filtration $\{R(X_2,t)\}_{t\geq 0}$ at $t=0,2,4$ is shown in Figure~\ref{fig_VR_cplxes_of_3_pts}, bottom. The corresponding bigraded barcodes are shown in Figure~\ref{fig_big_barcode_for_3_pts}.  We notice that the two bars in the top levels of the bigraded barcodes are identical. These two bars represent the ordinary (single graded) barcodes of $X_1$ and $X_2$, respectively. The two data sets $X_1$ and $X_2$ are distinguished by their bigraded barcodes. 

\begin{figure}[ht]
\begin{tikzpicture}
\begin{scope}[scale=0.5]
\node at (-1.5, 1) {$X_1$};
\draw[-stealth] (-0.5,0)--(2.5,0);
\draw[-stealth] (0,-0.5)--(0,2.5);
\draw[fill=blue, blue] (0,0) circle (3pt);
\draw[fill=blue, blue] (1,0) circle (3pt);
\draw[fill=blue, blue] (0,2) circle (3pt);

\begin{scope}[xshift=110]
\draw[-stealth] (-0.5,0)--(2.5,0);
\draw[-stealth] (0,-0.5)--(0,2.5);
\draw[fill=blue, blue] (0,0) circle (3pt);
\draw[fill=blue, blue] (1,0) circle (3pt);
\draw[fill=blue, blue] (0,2) circle (3pt);
\draw[very thick, blue] (0,0)--(1,0);
\end{scope}

\begin{scope}[xshift=220]
\draw[-stealth] (-0.5,0)--(2.5,0);
\draw[-stealth] (0,-0.5)--(0,2.5);
\draw[fill=blue, blue] (0,0) circle (3pt);
\draw[fill=blue, blue] (1,0) circle (3pt);
\draw[fill=blue, blue] (0,2) circle (3pt);
\draw[very thick, blue] (0,0)--(1,0);
\draw[very thick, blue] (0,0)--(0,2);
\end{scope}

\begin{scope}[xshift=330]
\draw[-stealth] (-0.5,0)--(2.5,0);
\draw[-stealth] (0,-0.5)--(0,2.5);
\draw[fill=blue, blue] (0,0) circle (3pt);
\draw[fill=blue, blue] (1,0) circle (3pt);
\draw[fill=blue, blue] (0,2) circle (3pt);
\draw[very thick, blue, fill=blue!30] (0,0)--(1,0)--(0,2)--cycle;
\end{scope}

\begin{scope}[yshift=-120]
\node at (-1.5, 1) {$X_2$};
\draw[-stealth] (-0.5,0)--(2.5,0);
\draw[-stealth] (0,-0.5)--(0,2.5);
\draw[fill=red, red] (0,0) circle (3pt);
\draw[fill=red, red] (1,0) circle (3pt);
\draw[fill=red, red] (1/2,1.7320508076) circle (3pt);

\begin{scope}[xshift=110]
\draw[-stealth] (-0.5,0)--(2.5,0);
\draw[-stealth] (0,-0.5)--(0,2.5);
\draw[fill=red, red] (0,0) circle (3pt);
\draw[fill=red, red] (1,0) circle (3pt);
\draw[fill=red, red] (1/2,1.7320508076) circle (3pt);
\draw[very thick, red] (0,0)--(1,0);
\end{scope}

\begin{scope}[xshift=220]
\draw[-stealth] (-0.5,0)--(2.5,0);
\draw[-stealth] (0,-0.5)--(0,2.5);
\draw[fill=red, red] (0,0) circle (3pt);
\draw[fill=red, red] (1,0) circle (3pt);
\draw[fill=red, red] (1/2,1.7320508076) circle (3pt);
\draw[very thick, red, fill=red!20] (0,0)--(1/2,1.7320508076)--(1,0)--cycle;
\end{scope}
\end{scope}
\end{scope}
\end{tikzpicture}
\caption{Two sequences of Vietoris--Rips complexes.}
\label{fig_VR_cplxes_of_3_pts}
\end{figure}
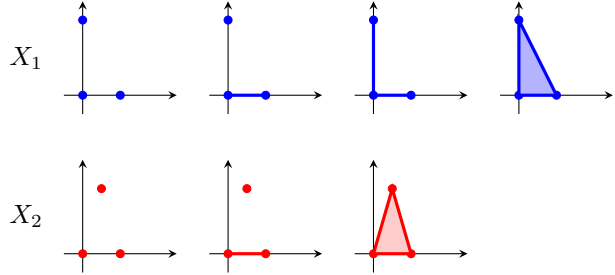

\begin{figure}
\begin{tikzpicture}[scale=0.65, yscale=0.8]
\draw[-stealth] (xyz cs:x=0) -- (xyz cs:x=7) node[right] {\scriptsize$t$};
\draw[-stealth] (xyz cs:y=0) -- (xyz cs:y=4.5) node[right] {\scriptsize$2j$};
\draw[-stealth] (xyz cs:z=0) -- (xyz cs:z=4.5) node[left] {\scriptsize$-i$};

\node[below] at (2.2,0,0) {\scriptsize$2$};
\node[below] at (4.4,0,0) {\scriptsize$4$};
\node[below] at (6.6,0,0) {\scriptsize$2\sqrt{5}$};

\draw[blue, thick] (0,3.4,2.5)--(4.2,3.4,2.5); 
\draw[fill=blue, blue] (0,3.4,2.5) circle (1.5pt); 
\draw[fill=blue, blue] (2.1,3.4,2.5) circle (1.5pt);

\draw[blue, thick] (0,3.6,2.5)--(2.1,3.6,2.5);
\draw[fill=blue, blue] (0,3.6,2.5) circle (1.5pt); 

\draw[blue, thick] (0,2.7,1.5)--(2.1,2.7,1.5);
\draw[fill=blue, blue] (0,2.7,1.5) circle (1.5pt);

\draw[blue, thick] (0,2.5,1.5)--(4.2,2.5,1.5);
\draw[fill=blue, blue] (0,2.5,1.5) circle (1.5pt);
\draw[fill=blue, blue] (2.1,2.5,1.5) circle (1.5pt);

\draw[blue, thick] (0,2.3,1.5)--(6.3,2.3,1.5);
\draw[fill=blue, blue] (0,2.3,1.5) circle (1.5pt);
\draw[fill=blue, blue] (2.1,2.3,1.5) circle (1.5pt);
\draw[fill=blue, blue] (4.2,2.3,1.5) circle (1.5pt);

\draw[blue, thick] (0,0.5,0.5)--(7,0.5,0.5);
\draw[fill=blue, blue] (0,0.5,0.5) circle (1.5pt);
\draw[fill=blue, blue] (2.1,0.5,0.5) circle (1.5pt);
\draw[fill=blue, blue] (4.2,0.5,0.5) circle (1.5pt);
\draw[fill=blue, blue] (6.3,0.5,0.5) circle (1.5pt);

\foreach \coo in {1,2,3,4}
	{\draw[dotted] (0,\coo,0)--(0,\coo,4);
	\draw[dotted] (0,0,\coo)--(0,4,\coo);
	}
	
\begin{scope}[xshift=60]	
\foreach \coo in {0,1,2,3,4}
	{\draw[dotted] (0,\coo,0)--(0,\coo,4);
	\draw[dotted] (0,0,\coo)--(0,4,\coo);
	}
\end{scope}

\begin{scope}[xshift=120]	
\foreach \coo in {0,1,2,3,4}
	{\draw[dotted] (0,\coo,0)--(0,\coo,4);
	\draw[dotted] (0,0,\coo)--(0,4,\coo);
	}
\end{scope}

\begin{scope}[xshift=180]	
\foreach \coo in {0,1,2,3,4}
	{\draw[dotted] (0,\coo,0)--(0,\coo,4);
	\draw[dotted] (0,0,\coo)--(0,4,\coo);
	}
\end{scope}

\begin{scope}[xshift=300]
\draw[-stealth] (xyz cs:x=0) -- (xyz cs:x=6) node[right] {\scriptsize$t$};
\draw[-stealth] (xyz cs:y=0) -- (xyz cs:y=4.5) node[right] {\scriptsize$2j$};
\draw[-stealth] (xyz cs:z=0) -- (xyz cs:z=4.5) node[left] {\scriptsize$-i$};

\node[below] at (2.2,0,0) {\scriptsize$2$};
\node[below] at (4.4,0,0) {\scriptsize$4$};

\draw[red, thick] (0,3.4,2.5)--(4.2,3.4,2.5); 
\draw[fill=red, red] (0,3.4,2.5) circle (1.5pt); 
\draw[fill=red, red] (2.1,3.4,2.5) circle (1.5pt);

\draw[red, thick] (0,3.6,2.5)--(2.1,3.6,2.5);
\draw[fill=red, red] (0,3.6,2.5) circle (1.5pt); 

\draw[red, thick] (0,2.7,1.5)--(2.1,2.7,1.5);
\draw[fill=red, red] (0,2.7,1.5) circle (1.5pt);

\draw[red, thick] (0,2.5,1.5)--(4.2,2.5,1.5);
\draw[fill=red, red] (0,2.5,1.5) circle (1.5pt);
\draw[fill=red, red] (2.1,2.5,1.5) circle (1.5pt);

\draw[red, thick] (0,2.3,1.5)--(4.2,2.3,1.5);
\draw[fill=red, red] (0,2.3,1.5) circle (1.5pt);
\draw[fill=red, red] (2.1,2.3,1.5) circle (1.5pt);

\draw[red, thick] (0,0.5,0.5)--(6,0.5,0.5);
\draw[fill=red, red] (0,0.5,0.5) circle (1.5pt);
\draw[fill=red, red] (2.1,0.5,0.5) circle (1.5pt);
\draw[fill=red, red] (4.2,0.5,0.5) circle (1.5pt);

\foreach \coo in {1,2,3,4}
	{\draw[dotted] (0,\coo,0)--(0,\coo,4);
	\draw[dotted] (0,0,\coo)--(0,4,\coo);
	}
	
\begin{scope}[xshift=60]	
\foreach \coo in {0,1,2,3,4}
	{\draw[dotted] (0,\coo,0)--(0,\coo,4);
	\draw[dotted] (0,0,\coo)--(0,4,\coo);
	}
\end{scope}

\begin{scope}[xshift=120]	
\foreach \coo in {0,1,2,3,4}
	{\draw[dotted] (0,\coo,0)--(0,\coo,4);
	\draw[dotted] (0,0,\coo)--(0,4,\coo);
	}
\end{scope}

\end{scope}
\end{tikzpicture}
\caption{The bigraded barcodes for 3 points in $\RR^2$.}
\label{fig_big_barcode_for_3_pts}
\end{figure}

\end{example}

\subsection{Double cohomology of moment-angle complexes}\label{subsec_hhzk}
In~\cite{l-p-s-s}, the homology of the moment-angle complex $H_*(\mathcal Z_K)=\bigoplus_{I\subset[m]}\widetilde H_*(K_I)$ was endowed with the second differential~$\partial'$. The 
homology of the resulting chain complex $\CH_*(\zk)=(H_*(\zk),\partial')$ was called the double homology of~$\zk$. We use this construction to define a new bigraded persistence module $\mathcal{PHHZ}(X)$.

Given $j\in[m]\setminus I$, consider the homomorphism
\[
  \phi_{p;I,j}\colon\widetilde H_p(K_I)\to \widetilde H_p(K_{I\sqcup\{j\}})
\]
induced by the inclusion $K_I\hookrightarrow K_{I\cup\{j\}}$. Then define
\[
  \partial'_p=(-1)^{p+1}\bigoplus_{I\subset[m],\,j\in[m]\setminus I}\varepsilon(j,I)\,\phi_{p;I,j},
\]
where 
\[
  \varepsilon(j,I)=(-1)^{\#\{i\in I\colon i<j\}}.
\]

The homomorphism $\partial'_p\colon\bigoplus_{I\subset[m]}\widetilde H_p(K_I)\to\bigoplus_{I\subset[m]}\widetilde H_p(K_I)$ satisfies $(\partial'_p)^2=0$. 
We therefore have a chain complex 
\begin{equation}\label{CH}
  \CH_\ast(\zk)\colonequals (H_\ast(\zk), \partial'),\quad
  \partial' \colon \widetilde{H}_{-i, 2j}(\zk) \to \widetilde{H}_{-i-1, 2j+2}(\zk),
\end{equation}
and bigraded \emph{double homology} of $\zk$:
\[
  \HH_\ast(\zk)=H(H_\ast(\zk), \partial').
\]

\emph{Double cohomology} of $\zk$ is defined similarly:
\[
  \HH^\ast(\zk)=H(H^\ast(\zk), d'),\quad
  d'\colon \widetilde{H}^{-i, 2j}(\zk) \to \widetilde{H}^{-i+1, 2j-2}(\zk)
\]
Double cohomology $\HH^\ast(\zk)$ can also be defined as the first double cohomology of the bicomplex
$\bigl(\Lambda[u_1,\ldots,u_m]\otimes
\mathbf k[K],d,d'\bigr)$ with $d=\sum_{i=1}^mv_i\frac\partial{\partial u_i}$, $d'=\sum_{i=1}^m\frac\partial{\partial u_i}$ and $dd'=-d'd$:
\[
  \HH^*(\zk)\cong H\bigl(H\bigl(\Lambda[u_1,\ldots,u_m]\otimes
  \mathbf k[K],d\bigr),d'\bigr),
\]
see~\cite[Theorem~4.3]{l-p-s-s}.

If $K=\Delta^{m-1}$, the full simplex on~$[m]$, then double homology is trivial:
\[
  \HH_*(\zk)=\mathbf k_{(0,0)}
\]
(the latter means that $\HH_{0,0}(\zk)\cong\mathbf k$ and $\HH_{-i,2j}(\zk)=0$ if $(-i,2j)\ne(0,0)$).

An important property of double homology is that attaching a simplex to $K$ along a face (in particular, adding a disjoint simplex) destroys most of $\HH_*(\zk)$, as described next:

\begin{theorem}[{\cite[Theorem~6.7]{l-p-s-s}}]\label{surgery}
Let $K=K' \cup_I  \Delta^n$ be a simplicial complex obtained from a nonempty simplicial complex $K'$ by gluing an $n$-simplex, $n\ge0$, along a proper, possibly empty, face $I \in K'$. Then either $K$ is a simplex, or 
\[
  \HH_*(\zk)=\mathbf k_{(0,0)}\oplus\mathbf k_{(-1,4)}.
\]
\end{theorem}

We note that the above property for simplicial complexes obtained by attaching a simplex along a face is extended by Valenzuela Ruiz and Stanley \cite{SV-arXiv} to the family of simplicial complexes obtained by gluing two arbitrary simplicial complexes along a possibly empty proper face of each. Such a simplicial complex is called \emph{wedge-decomposible}. The readers are referred to \cite[Section 8.1]{l-p-s-s} and \cite[Theorem 5.4]{SV-arXiv} for further details.

\begin{definition}\label{def_outlier}
Let $(X, d_X)$ be a finite metric space. A point $x\in X$ is a \emph{strong outlier} if for any $y\ne x$ holds the inequality
\[
  d_X(x,y)\geq \max_{y'\neq x} d_X(y, y'). 
\]
\end{definition}

\begin{proposition}\label{HHout}
Suppose a finite metric space $(X, d_X)$ has a strong outlier. Let $D$ be the diameter of $(X, d_X)$, and let $R(X,t)$ be the Vietoris--Rips complex associated with $(X, d_X)$. Then, 
\[
\HH_\ast(\mathcal{Z}_{R(X,t)})=\begin{cases}
\mathbf k_{(0,0)} & t\geq D;\\
\mathbf k_{(0,0)}\oplus \mathbf k_{(-1,4)} & t<D
\end{cases}
\]
with all maps $\HH_\ast(\mathcal{Z}_{R(X,t)})\to \HH_\ast(\mathcal{Z}_{R(X,t')})$ being the identity if 
$D\leq t\leq t'$, and the projection if $t<D\leq t'$.
\end{proposition}
\begin{proof}
The definition of an outlier $x$ implies that the Vietoris--Rips complex $R(X,t)$ is obtained from the Vietoris--Rips complex $R(X\setminus x,t)$ by attaching a simplex (with vertices $x$ and $\{y\colon d(x,y)\le t\}$) along its facet (with vertices $\{y\colon d(x,y)\le t\}$). The complex $R(X,t)$ is a simplex for $t\ge D$, and $R(X,t)$ is a simplicial complex obtained from $R(X\setminus x, t)$ by gluing a simplex along a proper face for $t<D$. Then the result follows from Theorem~\ref{surgery}.
\end{proof}

We define the bigraded persistent double homology and the corresponding barcodes by analogy with Definition~\ref{bgph}:

\begin{definition}
Let $(X,d_X)$ be a finite pseudo-metric space and $\{R(X,t)\}_{t\geq 0}$ its associated Vietoris--Rips filtration.
The \emph{persistent double homology} module of bidegree $(-i,2j)$ is 
\[
  \mathcal{PHHZ}_{-i,2j}(X) \colon \RR_{\geq 0} \to \mathbf{k}\ts{-mod},\quad
  t\mapsto \HH_{-i,2j}(\mathcal Z_{R(X,t)}).
\]
It maps $t\in \RR_{\geq0}$ to the bigraded module  $\HH_{-i,2j}(\mathcal Z_{R(X,t)})$ and maps a morphism $t_1\leq t_2$ to the homomorphism $\HH_{-i,2j}(\mathcal Z_{R(X,t_1)})\to\HH_{-i,2j}(\mathcal Z_{R(X,t_2)})$ induced by the inclusion $\mathcal Z_{R(X,t_1)}\to \mathcal Z_{R(X,t_2)}$ (see~\cite[Proposition~4.8]{l-p-s-s}). We also define the bigraded persistent double homology module
\begin{equation}\label{eq_PHHZ}
 \mathcal{PHHZ}(X)=
 \bigoplus_{\substack{0\leq i\leq m-1; \\ 0\leq j \leq m}}\mathcal{PHHZ}_{-i,2j}(X).
\end{equation}

One can view the bigraded persistent homology module~\eqref{eq_PHZ} as a functor to differential bigraded $\mathbf{k}$-modules,
\[
  \mathcal{PHZ}(X)\colon\RR_{\geq0}\to
  \ts{dg}(\mathbf k\ts{-mod}), \quad
  t\mapsto (H_{*,*}(\mathcal Z_{R(X,t)}),\partial').
\]
Then we have
\begin{equation}\label{phhcomposite}
  \mathcal{PHHZ}(X)=\mathcal H\circ\mathcal{PHZ}(X),
\end{equation}
where $\mathcal H\colon\ts{dg}(\mathbf k\ts{-mod}) \to\mathbf k\ts{-mod}$ is the homology functor. This will be convenient when we compare interleaving distances.
We denote by $\BDB(X)$ the barcode corresponding to the bigraded persistence module $\mathcal{PHHZ}(X)$.
\end{definition}

\begin{example}
Suppose $(X,d_X)$ has a strong outlier as in Definition \ref{def_outlier}. Then $\BDB(X)$ consists of an infinite bar $[0,\infty)$ at bidegree $(0,0)$ and a bar $[0,D)$ at bidegree $(-1,4)$, by Proposition~\ref{HHout}.
\end{example}

\begin{example}\label{ex_wedge_of_squares}
Let $G$ be the graph described below. 
\[
\begin{tikzpicture}[scale=0.8]
    \draw[fill] (0,0) circle (2pt);
    \draw[fill] (2,0) circle (2pt);
    \draw[fill] (1,1) circle (2pt);
    \draw[fill] (1,-1) circle (2pt);
    \draw[fill] (-2,0) circle (2pt);
    \draw[fill] (-1,1) circle (2pt);
    \draw[fill] (-1,-1) circle (2pt);
    \draw[thick] (0,0)--(1,1)--(2,0)--(1,-1)--(-1,1)--(-2,0)--(-1,-1)--(0,0);
    \node[above] at (-1,1) {$2$};
    \node[left] at (-2,0) {$1$};
    \node[below] at (-1,-1) {$3$};
    \node[above] at (0,0) {$4$};
    \node[above] at (1,1) {$5$};
    \node[right] at (2,0) {$7$};
    \node[below] at (1,-1) {$6$};
\end{tikzpicture}
\]
The associated finite pseudo-metric space $(X,d_X)$ is the set of vertices of $G$ with the geodesic distance, namely the distance between two vertices is the number of edges in a shortest path. For a subset $A \subset [7]$, we denote by $\Delta_A$ the simplex over the vertex set  $A$. Then, the Vietoris--Rips simplicial complex $R(X, t)$ associated with $(X, d_X)$ is given by 
\[
R(X,t)=\begin{cases}
X & 0\leq t < 1;\\
G & 1\leq t < 2;\\
\Delta_{\{1,2,3,4\}}\cup_{\Delta_{\{2,3,4\}}} \Delta_{\{2,3,4,5,6\}} \cup_{\Delta_{\{4,5,6\}}} \Delta_{\{4,5,6,7\}} & 2\leq t<3; \\
\Delta_{\{1,2,3,4,5,6\}} \cup_{\Delta_{\{2,3,4,5,6\}}} \Delta_{\{2,3,4,5,6,7\}} & 3\leq t < 4;\\
\Delta_{[7]} & 4\leq t.
\end{cases}
\]
In this case, the first four simplicial complexes have double homology $\HH_\ast(\mathcal{Z}_{R(X,t)})$ isomorphic to $\mathbf{k}_{(0,0)}\oplus \mathbf{k}_{(-1,4)}$ by Theorem \ref{surgery} and the result of \cite[Theorem 5.4]{SV-arXiv}. Hence, $\mathbb{BB}(X)$ consists of an infinite bar $[0, \infty)$ at bidegree $(0,0)$ and a bar $[0,4)$ at bidegree $(-1,4)$. 
\end{example}

In Example \ref{ex_wedge_of_squares}, if we only consider the subset $S=\{1,2,3,4\}$ of $X$, then the corresponding Vietoris--Rips simplicial complex $R(S,t)$ is given by $S$ for $0\leq t <1$, the $4$-cycle for $1\leq t< 2$, and $\Delta_{[4]}$ for $2\leq t$. Hence, we get the more interesting bigraded barcode $\mathbb{BB}(S)$ with $[0, \infty)$ in bidegree $(0,0)$, $[0,2)\oplus [1,2)$ in bidegree $(-1,4)$ and $[1,2)$ in bidegree $(-2, 8)$; the readers are referred to Theorem \ref{surgery} and \cite[Section 7]{l-p-s-s}. Essentially after removing the points $\{5,6,7\}$, we can see the circle, but with all of $X$ we only see 
two points $\{1,7\}$ through $\PHHZ (X)$. Similarly if we have a strong outlier, we only see two points, but removing it could give interesting bigraded barcode of $\mathcal{PHHZ}$. Perhaps, the bigraded barcode of $\mathcal{PHHZ}$ is sensitive to detecting minimal subcomplexes that carry persistent homology classes.

The following example shows that the persistent double homology module might detect different local structure of a given filtration. 

\begin{example}
Consider two filtrations of graphs shown in Figure \ref{fig_two_filtrations}, which could arise from a single graph with different weights on edges. Starting with~5 discrete points, each of the second terms is homotopy equivalent to the disjoint union of a circle and a point. The third terms are homotopy equivalent to a wedge of two circles. Hence, the usual persistent homology module $\mathcal{PH}$ does not distinguish between these filtrations. However, the bigraded persistent double homology module $\mathcal{PHHZ}$ gives different signals in its bigraded barcode.

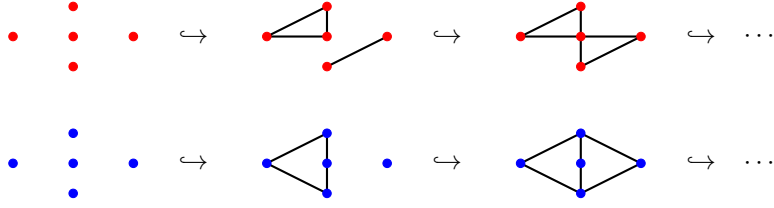
\begin{figure}
\begin{tikzpicture}[scale=0.8]
\draw[fill=red, red] (1,0) circle (2pt);
\draw[fill=red, red] (-1,0) circle (2pt);
\draw[fill=red, red] (0,1/2) circle (2pt);
\draw[fill=red, red] (0,-1/2) circle (2pt);
\draw[fill=red, red] (0,0) circle (2pt);
\node at (2,0) {$\hookrightarrow$};
\begin{scope}[xshift=120]
\draw[thick] (0,0)--(-1,0)--(0,1/2)--cycle;
\draw[thick] (0,-1/2)--(1,0);
\draw[fill=red, red] (1,0) circle (2pt);
\draw[fill=red, red] (-1,0) circle (2pt);
\draw[fill=red, red] (0,1/2) circle (2pt);
\draw[fill=red, red] (0,-1/2) circle (2pt);
\draw[fill=red, red] (0,0) circle (2pt);
\node at (2,0) {$\hookrightarrow$};
\end{scope}
\begin{scope}[xshift=240]
\draw[thick] (-1,0)--(0,1/2)--(0,0)--cycle;
\draw[thick] (0,0)--(1,0)--(0,-1/2)--cycle;
\draw[fill=red, red] (1,0) circle (2pt);
\draw[fill=red, red] (-1,0) circle (2pt);
\draw[fill=red, red] (0,1/2) circle (2pt);
\draw[fill=red, red] (0,-1/2) circle (2pt);
\draw[fill=red, red] (0,0) circle (2pt);
\node at (2,0) {$\hookrightarrow$};
\node at (3,0) {$\cdots$};
\end{scope}

\begin{scope}[yshift=-60]
\draw[fill=blue, blue] (1,0) circle (2pt);
\draw[fill=blue, blue] (-1,0) circle (2pt);
\draw[fill=blue, blue] (0,1/2) circle (2pt);
\draw[fill=blue, blue] (0,-1/2) circle (2pt);
\draw[fill=blue, blue] (0,0) circle (2pt);
\node at (2,0) {$\hookrightarrow$};
\begin{scope}[xshift=120]
\draw[thick] (-1,0)--(0,1/2)--(0,-1/2)--cycle;
\draw[fill=blue, blue] (1,0) circle (2pt);
\draw[fill=blue, blue] (-1,0) circle (2pt);
\draw[fill=blue, blue] (0,1/2) circle (2pt);
\draw[fill=blue, blue] (0,-1/2) circle (2pt);
\draw[fill=blue, blue] (0,0) circle (2pt);
\node at (2,0) {$\hookrightarrow$};
\end{scope}
\begin{scope}[xshift=240]
\draw[thick] (1,0)--(0,1/2)--(-1,0)--(0,-1/2)--cycle;
\draw[thick] (0,1/2)--(0,-1/2);
\draw[fill=blue, blue] (1,0) circle (2pt);
\draw[fill=blue, blue] (-1,0) circle (2pt);
\draw[fill=blue, blue] (0,1/2) circle (2pt);
\draw[fill=blue, blue] (0,-1/2) circle (2pt);
\draw[fill=blue, blue] (0,0) circle (2pt);
\node at (2,0) {$\hookrightarrow$};
\node at (3,0) {$\cdots$};
\end{scope}
\end{scope}
\end{tikzpicture}
\caption{Two filtrations of graphs.}
\label{fig_two_filtrations}
\end{figure}

Indeed, in the first filtration, each of the simplicial complexes has the double homology isomorphic to $\mathbf{k}_{(0,0)}\oplus \mathbf{k}_{(-1,4)}$ by \cite[Theorem 5.4]{SV-arXiv}; see also \cite[Example 8.7]{l-p-s-s} for a detailed calculation. Hence, the corresponding bigraded barcode begins with two bars and persists up to the parameter corresponding to the third term. 

In the second filtration, the second term has the double homology isomorphic to $\mathbf{k}_{(0,0)}\oplus \mathbf{k}_{(-1,4)}$ by Theorem \ref{surgery}. However, the third term, say $K$, is the join of two disjoint points and three disjoint points. We note that the double homology of disjoint points is isomorphic to $\mathbf{k}_{(0,0)}\oplus \mathbf{k}_{(-1,4)}$. Moreover, the double homology $\HH_\ast(\mathcal{Z}_{K_1\ast K_2})$ corresponding to the join $K_1\ast K_2$ of two simplicial complexes $K_1$ and $K_2$ is isomorphic to the tensor product $\HH_\ast(\mathcal{Z}_{K_1})\otimes \HH_\ast(\mathcal{Z}_{K_2})$. We refer to \cite[Theorem 6.3]{l-p-s-s}. Therefore, combining these two results, it follows that  $\HH_*(\zk)$ is isomorphic to $\mathbf{k}_{(0,0)}\oplus \mathbf{k}^2_{(-1,4)}\oplus \mathbf{k}_{(-2,8)}$.
Hence, the corresponding bigraded barcode begins with two bars and extra two bars are born at the parameter for the third term. 
\end{example}

\subsection{Generalizations of $\HH$}

We give a generalization of $\HH$ that can be applied to any functor out of the poset category of a simplex. For this subsection, we ignore degrees. 

The poset category $P([m])$ of the simplex~$\Delta^{m-1}$ has objects $I\subset[m]$ and morphisms $I\subset J$.
For an abelian category $\mathcal A$ and a functor $F\colon P([m])\rightarrow \mathcal A$, define $C(F)\colonequals \bigoplus_{I\subset[m]} F(I)$, equipped with a differential $\partial'$ as before. In other words, for $x\in F(I)$  and $\phi_{I,j}\colon I\hookrightarrow I\cup \{j\}$ being the inclusions,
\[
   \partial'(x)=\sum_{I,j} \varepsilon(j,I) F(\phi_{I,j})(x).
\]   
If $j$ is already in $I$, then the term is~$0$. We take $H(F)$ to be the homology of the chain complex $(C(F),\partial')$. Taking a simplicial complex $K$, if we let $F(I)=\widetilde H_\ast(K_I)$ we get that $H(F)=\HH_*(\mathcal Z_K)$.
We learned about this generalization from Daisuke Kishimoto
(see also \cite{Cha}), and we can make a cohomology functor for contravarient functors similarly. A simple way to get other functors is to use functors out of spaces other than reduced homology.  See \cite{CCC} for some other variants of $\HH$. 

We give two examples of this generalization.
First, let $X$ be a space and $K$ a simplicial complex. Consider the functor $G\colon P([m]) \to \ts{Top}$ defined by 
\[
G(I)=\mathop \mathit{Map}(K_I,X),\]
where $\mathit{Map}(K_I,X)$ denotes the mapping space from $K_I$ to $X$. Composing this with a functor into an abelian category such as the homotopy groups $\pi_*$, we can then take the double homology. 

For our second example, we consider a multidimensional persistence module as a multigraded module
$$M=\{M_v\}_{v\in \mathbb{Z}^m}$$  
over the multigraded ring $R=\mathbf{k}[t_1, \dots, t_m]$.
For $I\subset [m]$, we let $R_I$ be obtained from $R$ by inverting $t_i$ when $i\in I$.
Consider $M_I=R_I\otimes_R M$. 

Now we can define the functor by $F(I)=M_I$. These functors are studied in~\cite{FS}. We can identify the inclusion of $\HF_{0}(M)$ in $M$ as the torsion coming from the torsion subcategory generated by $R/(t_1,\ldots,  t_n)$. The objects of this category are $M$ such that for every $x\in M$ and $j\in [m]$ there is an $n$ so that $t_j^nx=0$. The higher homology classes in $\HF_i(M)$ give more mysterious invariants of multi-dimensional persistence modules. 

In this context stability relates to doubling a variable $t_i$ to $t_i, t_i'$. The corresponding $\HH$ will be modules over $R$ and $R[t_i']$ and to compare them we might use the map of rings $R[t_i']\rightarrow R$ sending $t_i'$ to $t_i$. This corresponds to reducing two highly correlated variables to a single variable and gives a pullback functor on modules. There will be a subcategory of $R$ modules which are stable under a comparison map between the modules.

\section{Stability}\label{sec_stability_HZK}
In this section, we establish stability properties of the bigraded persistence modules $\mathcal{PHZ}(X)$ and $\mathcal{PHHZ}(X)$; see  Theorems \ref{thm_BB_stable} and \ref{thm_BB_HH_stable}.

\subsection{Distance functions}
We begin by summarizing several distance functions defined on the categories of finite pseudo-metric spaces and persistence modules. 

\begin{definition}\label{def_GH_distance}
The \emph{Hausdorff distance} between two nonempty subsets $A$ and $B$ in a finite pseudo-metric space $(Z,d)$ is 
\[
  d_H(A,B)\vcentcolon= \max\bigl\{\sup_{a\in A}d(a,B),\,\sup_{b\in B}d(A,b)\bigr\}.
\]
The \emph{Gromov--Hausdorff distance} between two finite pseudo-metric spaces $X$ and $Y$ is 
\[
  \dgh(X,Y)\vcentcolon= \inf_{Z, f, g}
  d_H(f(X), g(Y)),
\]
where the infimum is taken over all isometric embeddings $f\colon X\to Z$ and $g \colon Y \to Z$ into a pseudo-metric space~$Z$.  
\end{definition}

Using the notion of a correspondence between two sets, one can give an alternative definition of the Gromov--Hausdorff distance, which is often convenient for computational purposes. 

\begin{definition}\label{def_correspondence}
Given two sets $X$ and $Y$ and a subset $C$ of $X\times Y$, define $D_{C}(x)$ as the number of $y\in Y$ for which $(x,y)\in C$. Similarly, define $D_C(y)$ as the number of $x\in X$ for which $(x,y)\in C$.
A \emph{correspondence} between two sets $X$ and $Y$ is a subset $C$ of $X\times Y$ such that both $D_{C}(x)$ and $D_C(y)$ are nonzero for all $x \in X$ and $y\in Y$.
\end{definition}

Note that the condition $D_{C}(x)\geq 1$ for all $x\in C$ is equivalent to saying that $C$ is a (multi-valued) mapping with domain $X$ and codomain $Y$. Furthermore, the condition $D_{C}(y)\geq 1$ is equivalent to saying that $C$ is surjective. 

\begin{proposition}[{\cite[Theorem 7.3.25]{BBI}}]
For two finite pseudo-metric spaces $(X, d_X)$ and $(Y, d_Y)$, 
\begin{equation}\label{eq_GH_distance_corresp}
\dgh(X,Y)=\frac{1}{2}\min\limits_{C}\max\limits_{(x_{1},y_{1}),(x_{2},y_{2})\in C}|d_{X}(x_{1},x_{2})-d_{Y}(y_{1},y_{2})|.
\end{equation}
\end{proposition}

We have the following modification of the Gromov--Hausdorff distance~\eqref{eq_GH_distance_corresp}, to be used below.

\begin{definition}\label{d'gh}
For two finite pseudo-metric spaces $(X, d_X)$ and $(Y, d_Y)$ of the same cardinality, define $\dpgh(X,Y)$ by formula~\eqref{eq_GH_distance_corresp} in which the infimum is taken over bijective correspondences $C$ only.
\end{definition}

Obviously, $\dpgh(X,Y)\geq \dgh(X,Y)$ for any pair of pseudo-metric spaces of the same cardinality. The next example shows that for two pseudo-metric spaces $X$ and $Y$ of the same cardinality, the difference between $\dgh(X,Y)$ and $\dpgh(X,Y)$ can be arbitrarily large.

\begin{example}
Fix $n>2$. Let $X$ consist of four vertices of a rectangle in $\mathbb R^2$ with two edges of length~$1$ and diagonals of length~$n$. Let $Y$ consist of four vertices of a tetrahedron in $\mathbb R^3$ with the base an equilateral triangle with edges of length~1 and the other three edges of length~$n$. See Figure~\ref{fig_dgh_neq_d_prime_GH}. It is easy to see that $\dpgh(X,Y)=\frac{1}{2}(n-1)$, while $\dgh(X,Y)=\frac{1}{2}$.     
\end{example}

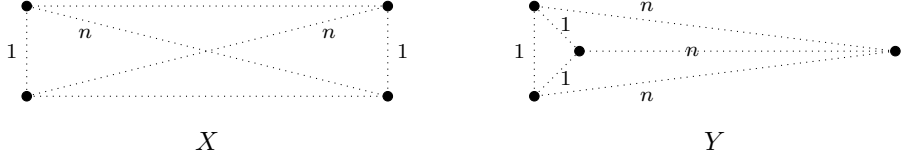
\begin{figure}
\begin{tikzpicture}[scale=0.6]
    \draw[dotted] (0,0)--(0,2)--(8,2)--(8,0)--cycle;
    \draw[dotted] (0,0)--(8,2);
    \draw[dotted] (0,2)--(8,0);
    \draw[fill] (0,2) circle (3pt);
    \draw[fill] (8,2) circle (3pt);
    \draw[fill] (8,0) circle (3pt);
    \draw[fill] (0,0) circle (3pt);
    \node[left] at (0,1) {\footnotesize$1$};
    \node[right] at (8,1) {\footnotesize$1$};
    \node[below] at (1.3,1.7) {\footnotesize$n$};
    \node[below] at (6.7,1.7) {\footnotesize$n$};
    \node at (4,-1) {$X$};
\begin{scope}[xshift=320]
    \draw[dotted] (0,0)--(1,1)--(0,2)--cycle;
    \draw[dotted] (8,1)--(0,0);
    \draw[dotted] (8,1)--(0,2);
    \draw[dotted] (8,1)--(1,1);

    \draw[fill] (0,0) circle (3pt);
    \draw[fill] (0,2) circle (3pt);
    \draw[fill] (1,1) circle (3pt);
    \draw[fill] (8,1) circle (3pt);

    \node[left] at (0,1) {\footnotesize$1$};
    \node at (0.7, 1.6) {\footnotesize$1$};
    \node at (0.7,0.4) {\footnotesize$1$};

    \node at (2.5,2) {\footnotesize$n$};
    \node at (3.5,1) {\footnotesize$n$};
    \node at (2.5,0) {\footnotesize$n$};

    \node at (4,-1) {$Y$};

\end{scope}
\end{tikzpicture}
\caption{Pseudo-metric spaces with $d_{GH}(X,Y)\neq d'_{GH}(X, Y).$}
\label{fig_dgh_neq_d_prime_GH}
\end{figure}

Next, we introduce a distance function for persistence modules. 

Given a persistence module $\mathcal{M}$ as in \eqref{eq_persistent_module} and an arbitrary $\epsilon\in \RR_{\geq 0}$, we define  
\[
\mathcal{M}^\epsilon\colon \RR_{\geq 0} \to \mathbf{k}\ts{-mod}
\]
by $\mathcal{M}^\epsilon(s)=\mathcal{M}(s+\epsilon)$ and $\mathcal{M}^\epsilon(\phi_{s_1,s_2})=
\phi_{s_1+\epsilon, s_2+\epsilon}$. There is a natural transformation 
\[
\eta^\epsilon_\mathcal{M} \colon \mathcal{M} \to \mathcal{M}^\epsilon
\]
defined by the family of morphisms $\{ M_{s} \to M_{s+\epsilon} \mid s\in \mathbb{R}_{\geq 0}\}$. 

Given two persistence modules $\mathcal{M}$ and $\mathcal{N}$, we say that $\mathcal{M}$ and $\mathcal{N}$ are \emph{$\epsilon$-interleaved} if there exist natural transformations  
\[
\beta \colon \mathcal{M} \to \mathcal{N}^\epsilon\quad \text{ and }\quad \gamma \colon \mathcal{N} \to \mathcal{M}^\epsilon
\]
such that $\beta^\epsilon \circ \gamma = \eta_\mathcal{N}^{2\epsilon}$ and $\gamma^\epsilon \circ \beta = \eta_\mathcal{M}^{2\epsilon}$, 
where $\beta^\epsilon \colon \mathcal{M}^\epsilon \to \mathcal{N}^{2\epsilon}$ and $\gamma^\epsilon \colon \mathcal{N}^\epsilon \to \mathcal{M}^{2\epsilon}$ denote the $\epsilon$-shifts of $\beta$ and $\gamma$, respectively.

\begin{definition}
The \emph{interleaving distance} between two persistence modules $\mathcal{M}$ and $\mathcal{N}$ is
\[
d_{\IL}(\mathcal{M}, \mathcal{N})\colonequals \inf \{\epsilon \in \RR_{\geq 0} \mid \mathcal{M} \text{ and } \mathcal{N} \text{ are $\epsilon$-interleaved}\}.
\]
\end{definition}

\begin{example}\label{ex-ild}
The interleaving distance between two interval modules $\mathbf{k}\left([a,b)\right)$ and $\mathbf{k}\left([a',b')\right)$ (see Example \ref{ex_interval_module}) is  given by 
\[
d_{\IL} \big(\mathbf{k}\left([a,b)\right), \mathbf{k}\left([a',b')\right) \big) = \min \left\{ \max \left\{ \frac{b-a}{2}, \frac{b'-a'}{2}\right\}, \max\{|a'-a|, |b'-b|\}\right\}.
\]
This formula has the following meaning. If the intervals are close to each other (the closure of each interval contains the midpoint of the other), then the distance is the maximum of the distances between their endpoints, i.\,e. $\max\{|a'-a|, |b'-b|\}$, which is the standard $l^\infty$-distance between the two intervals. If the intervals are far apart, then the distance is $\max \left\{ \frac{b-a}{2}, \frac{b'-a'}{2}\right\}$, where $\frac{b-a}{2}$ is the $l^\infty$-distance between $[a,b)$ and a `zero length interval' at the midpoint of~$[a,b)$.

The zero persistence module can be thought of as the interval module $\mathbf{k}(\varnothing)$ of an empty interval. In this case, the interleaving distance is given by
\[
  d_{\IL}\big(\mathbf{k}\left([a,b)\right), \mathbf{k}(\varnothing) \big) =\frac{b-a}2.
\]
\end{example}

Below, we record the following properties of $d_{\IL}$ for future use.

\begin{proposition}[{see~\cite[Proposition~5.5]{CSGO}}]
\label{oplus}
Let $\mathcal M_1$, $\mathcal M_2$, $\mathcal N_1$ and $\mathcal N_2$ be persistence modules. Then
\[
  d_{\IL}(\mathcal M_1\oplus\mathcal M_2,
  \mathcal N_1\oplus\mathcal N_2)\le
  \max\bigl(d_{\IL}(\mathcal M_1,\mathcal N_1), d_{\IL}(\mathcal M_2,\mathcal N_2)\bigr).
\]
\end{proposition}

The interleaving distance can be defined for persistence modules $\mathcal M\colon\RR_{\geq0}\to\ts{c}$ taking values in a category $\ts{c}$ more general than $\mathbf{k}\ts{-mod}$, see~\cite{b-s-s}. The following result is clear from the definition.

\begin{proposition}\label{prop_d_IL_composition}
Given two functors $\mathcal{M}$ and $\mathcal{N}$ from $\RR_{\geq 0}$ to $\ts{c}$, and a functor $\mathcal{F}\colon \ts{c} \to \ts{d}$, we have:
\begin{equation*}
d_{\IL}(\mathcal{F}\circ \mathcal{M}, \mathcal{F}\circ \mathcal{N}) \leq d_{\IL}(\mathcal{M}, \mathcal{N}).
\end{equation*}
\end{proposition}

Next, we introduce the bottleneck distance between multisets of intervals (or barcodes). It is defined in the standard way as the infimum over matchings between them, in which some of the intervals can be matched to zero length intervals at their midpoints.

Let $B$ and $B'$ be finite multisets of intervals of the form $[a,b)$ with $a\in\RR_{\ge0}$ and $b\in\RR_{\geq 0} \cup \{\infty\}$. Define the multiset $\overline{B}=B\cup\varnothing^{|B'|}$, obtained by adding to $B$ the multiset containing the empty interval $\varnothing$ with cardinality~$|B'|$. Similarly, define $\overline{B'}=B'\cup\varnothing^{|B|}$. Note that $\overline{B}$ and $\overline{B'}$ have the same cardinality. We define the distance function 
\[
\pi \colon \overline{B}\times \overline{B'} \to \RR_{\geq 0}\cup \{\infty\}
\]
by the formulae:
\begin{gather*}
  \pi\left( [a, b), [a', b')\right)= \max\{|a'-a|, |b'-b| \},\quad \pi([a,\infty),[a',\infty))=|a'-a|,\\
  \pi([a,b),\varnothing)=\frac{b-a}2,\quad
  \pi(\varnothing,[a',b'))=\frac{b'-a'}2,\quad
  \pi(\varnothing,\varnothing)=0,\\
  \pi([a,\infty),[a',b'))=\pi([a,b),[a',\infty))=
  \pi([a,\infty),\varnothing)=\pi(\varnothing,[a',\infty))=\infty
\end{gather*}
for finite $a,b,a',b'$. 

\begin{definition}\label{def_Wasserstein}
Denote by $\mathcal{D}\left(\overline{B}, \overline{B'}\right)$ the set of all bijections $\theta \colon \overline{B} \to  \overline{B'}$. 
Then the \emph{$\infty$-Wasserstein distance}, also called the \emph{bottleneck distance}, is defined as follows:
\[
  W_\infty(B, B')=\min_{\theta \in \mathcal{D}(\overline{B}, \overline{B'})} \max_{ I \in \overline{B}} \, \pi(I, \theta(I)).
\]
It can be shown~\cite[Definition~4.11]{BuSc} that
\[
  W_\infty(B, B')=\min_{\theta \in \mathcal{D}(\overline{B}, \overline{B'})} \max_{ I \in \overline{B}} \, d_{\IL}\bigl(\mathbf k(I),\mathbf 
  k( \theta(I))\bigr),
\]
where the interleaving distance between interval modules can be expressed via the $l^\infty$-distance as in Example~\ref{ex-ild}.
\end{definition}

\begin{example}\ \

\noindent\textbf{1.} Let $B=\{[0,5)\}$ and $B'=\{[1,7)\}$. Then $\overline B=\{[0,5),\varnothing\}$ and $\overline {B'}=\{[1,7),\varnothing\}$. There are two bijections between $\overline B$ and $\overline{B'}$:
\[
\theta_1\colon[0,5)\mapsto[1,7),\,\varnothing\mapsto\varnothing,\quad\text{and}\quad\theta_2\colon[0,5)\mapsto\varnothing,\,\varnothing\mapsto[1,7).
\]
The minimum of $\max_{ I \in \overline{B}} \, \pi(I, \theta(I))$ is achieved at $\theta_1$, which reflects the fact that the two intervals overlap sufficiently. We have $W_\infty(B, B')=\pi([0,5),[1,7))=7-5=2$.

\smallskip

\noindent\textbf{2.} Let $B=\{[0,5)\}$ and $B'=\{[3,9)\}$. These two intervals are `far apart', so the minimum of $\max_{ I \in \overline{B}} \, \pi(I, \theta(I))$ is achieved when both $[0,5)$ and $[3,9)$ are matched to an empty interval. We have $W_\infty(B, B')=\pi(\varnothing,[3,9))=3$.
\end{example}

In the literature, the following result is often referred to as the \emph{isometry theorem}, see \cite[Theorem 3.4]{Les}. We also refer to \cite[Theorem 4.16]{BuSc} and \cite[Theorem 5.14]{CSGO} for different versions of this result:

\begin{theorem}\label{prop_isometry}
Let $\mathcal{M}$ and $\mathcal{N}$ be persistence modules satisfying the hypothesis of~Theorem~\ref{prop_decomp_persitent_module}. Then, 
\[
d_{\IL}(\mathcal{M}, \mathcal{N})=W_\infty(B(\mathcal{M}), B(\mathcal{N})),
\]
where $B(\mathcal{M})$ and $B(\mathcal{N})$ are the barcodes corresponding to $\mathcal{M}$ and $\mathcal{N}$, respectively. 
\end{theorem}

The \emph{stability theorem} asserts that the persistent homology barcodes are stable under perturbations of the data sets in the Gromov--Hausdorff metric. It is a key result justifying the use of persistent homology in data science. We refer to~\cite[Theorem~3.1]{Ch} and~\cite[Lemma~4.3, Theorem~5.2]{c-d-o14} for the proof:

\begin{theorem}\label{stabph}
Let $(X,d_X)$ and $(Y,d_Y)$ be two finite pseudo-metric spaces, and let $B(X)$ and $B(Y)$ be the barcodes corresponding to the persistence modules $\PH(X)$ and $\PH(Y)$. Then,
\[
  d_{\IL}(\PH(X),\PH(Y))\le 2d_{GH}(X,Y),
\]
and when working over a field, 
\[
  W_\infty(B(X),B(Y))\le 2d_{GH}(X,Y).
\]
\end{theorem}

\subsection{The doubling operation}
The following construction plays a key role in proving a stability property of persistent double homology.

\begin{definition}\label{def_doubling_metric_sp}
Given a pseudo-metric space $(X, d_X)$ and a point $x\in X$, we refer to the pseudo-metric space $X^{\prime}=X\sqcup\{x^{\prime}\}$ with  $d_{X^{\prime}}(x,x^{\prime})=0$ as the \emph{doubling} of $X$ at~$x$.

Given a simplicial complex $K$ on $[m]$ and a vertex $\{i\}$ of $K$, the \emph{doubling} of $K$ at~$\{i\}$ is the minimal simplicial complex $K'$ on the set $[m]\sqcup\{i'\}$ which contains~$K$ and all subsets $I\sqcup \{i'\}$, where $i\in I\in K$. There is a deformation retraction $K'\to K$ sending $\{i'\}$ to $\{i\}$ and identity on other vertices. See Figure~\ref{fig_doubl}.
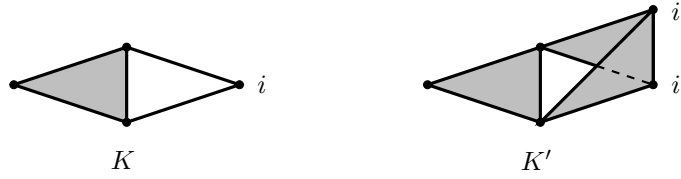
\begin{figure}[h]
\begin{tikzpicture}[scale=0.5]
\draw[very thick, fill=gray!50] (0,1)--(3,2)--(3,0)--cycle;
\draw[very thick] (3,2)--(6,1)--(3,0);
\draw[fill] (0,1) circle (3pt);
\draw[fill] (3,2) circle (3pt);
\draw[fill] (6,1) circle (3pt) 
node[anchor=west]{\ $i$};
\draw[fill] (3,0) circle (3pt);
\node at (2.9,-1) {$K$};
\draw[very thick, fill=gray!50] (11,1)--(14,2)--(14,0)--cycle;
\draw[very thick, fill=gray!50] (14,2)--(17,3)--(17,1)--cycle;
\draw[very thick, fill=gray!50] (14,0)--(17,3)--(17,1)--cycle;
\draw[thick, dashed] (15.5,1.5)--(17,1);
\draw[fill] (11,1) circle (3pt);
\draw[fill] (14,2) circle (3pt);
\draw[fill] (17,1) circle (3pt) 
node[anchor=west]{\ $i$};
\draw[fill] (17,3) circle (3pt) 
node[anchor=west]{\ $i'$};
\draw[fill] (14,0) circle (3pt);
\node at (13.9,-1) {$K'$};
\end{tikzpicture}
\caption{Doubling at a vertex.}
\label{fig_doubl}
\end{figure}

The doubling at a vertex is a particular case of the following operation introduced in~\cite{ab-pa19}.
Let $K$ be a simplicial complex on the set~$[m]$, and let $K_1, \ldots, K_m$ be simplicial complexes. The \emph{substitution} of $K_1, \ldots, K_m$ into $K$ is the simplicial complex
\[
  K(K_1, \ldots, K_m) = \{I_{j_1}\sqcup\dots\sqcup I_{j_k}\colon I_{j_l}\in K_{j_l},\; l = 1,\dots, k \quad\text{and}\quad \{j_1, \dots, j_k\}\in K\}.
\]
The doubling of $K$ at $\{i\}$ is the substitution complex $K(1,\ldots,\Delta[i,i'],\ldots,m)$, where~$\Delta[i,i']$ denotes the 1-simplex on $i$ and $i'$.

Observe that if $X'$ is the doubling of a pseudo-metric space $X$ at a point, then the Vietoris--Rips complex $R(X',t)$ is the doubling of $R(X,t)$ at a vertex, i.\,e. $R(X',t)=R(X,t)'$.
\end{definition}

\begin{proposition}\label{prop_d_GH_X_X'=0}
Suppose $\widehat{X}$ is obtained from $X$ by an arbitrary number of doubling operations performed at arbitrary points on the way. Then $d_{GH}(X,\widehat{X})=0$.
\end{proposition}

\begin{proof}
Using the triangle inequality, we reduce the claim to the case when $\widehat X$ is obtained from $X$ by a single doubling operation, i.\,e. $\widehat X=X'=X\sqcup \{x_0'\}$ is a doubling at $x_0\in X$. 

We use~\eqref{eq_GH_distance_corresp} and consider the following correspondence $C$ between $X$ and $X^{\prime}$:
\[
C \colonequals \{(x,x)\in  X\times X' \mid x\in X\}\sqcup \{(x_0,x_0')\}.
\]
Note that for any $s\in X$ with $ s\neq x_{0}$, the number $D_{C}(s)$ of $x
\in X'$ for which $(s,x)\in C$ is $1$. Similarly, $D_{C}(x_{0})=2$ and $D_{C}(t)=1$ for all $t\in X^{\prime}$. See Definition \ref{def_correspondence}. The triangle inequality together with $d_{X'}(x_0,x_0')=0$ implies $d_{X}(x_{1},x_{2})=d_{X^{\prime}}(y_{1},y_{2})$ for any $(x_{1},y_{1}),(x_{2},y_{2})\in C$, hence $d_{GH}(X,X^{\prime})=0$.
\end{proof}

\begin{proposition}\label{prop_dGH_d_prime_GH}
Given two finite pseudo-metric spaces $X$ and $Y$, 
there exist two finite pseudo-metric spaces $\widehat{X}$ and $\widehat{Y}$ such that $|\widehat X|=|\widehat Y|$ and 
\[
d_{GH}(X,Y)=d^{\prime}_{GH}(\widehat{X},\widehat{Y}),
\]
where $d_{GH}^{\prime}$ is the modified Gromov--Hausdorff distance, see Definition \ref{d'gh}.
\end{proposition}
\begin{proof}
For two finite pseudo-metric spaces $S$ and $T$ of the same cardinality, we have 
\begin{equation}\label{eq_d_GH_leq_d'_GH}
d_{GH}(S,T)\leq d^{\prime}_{GH}(S,T),
\end{equation}
because any bijection is a correspondence between $S$ and~$T$. 

Now let $\widehat{X}$ and $\widehat{Y}$ be pseudo-metric spaces obtained by iterated doublings of $X$ and $Y$, respectively, so that $\widehat{X}$ and $\widehat{Y}$ have the same cardinality. The triangle inequality together with Proposition~\ref{prop_d_GH_X_X'=0} implies
\begin{align*}
\begin{split}
d_{GH}(X,Y)&\leq d_{GH}(X,\widehat{X})+d_{GH}(\widehat{X},\widehat{Y})+d_{GH}(\widehat{Y},Y) \\
&=d_{GH}(\widehat{X},\widehat{Y}) \leq d^{\prime}_{GH}(\widehat{X},\widehat{Y}),
\end{split}
\end{align*}
where the second equality follows from Proposition \ref{prop_d_GH_X_X'=0} and the last inequality follows from \eqref{eq_d_GH_leq_d'_GH}. 

To complete the proof, we construct iterated doublings $\widehat{X}$ and $\widehat{Y}$ such that $d'_{GH}(\widehat{X},\widehat{Y})\leq d_{GH}(X,Y)$.
Let $C$ be a correspondence which realizes $d_{GH}(X,Y)$. If $C$ is a bijection (in particular, $|X|=|Y|$), then there is nothing to prove. So we assume that $C$ is not a bijection. 

First, consider all the vertices $x\in X$ such that $D_{C}(x)>1$ and double each of them~$D_{C}(x)-1$ times. We get a new pseudo-metric space $\widehat{X}$ and a new correspondence~$C^{\prime}$ between $\widehat{X}$ and $Y$ which matches each double of $x$ with a single point of~$Y$. If~$C^{\prime}$ is a bijection, we set $\widehat{Y}=Y$ and we have 
\begin{align}\label{eq_d_GH_geq_d'_GH}
\begin{split}
d_{GH}(X,Y) &= \frac{1}{2} \max_{(x,y), (x', y')\in C} |  d_X(x,x')-d_Y(y,y') | \\
&= \frac{1}{2} \max_{(x,y), (x', y')\in C'} |  d_{\widehat{X}}(x,x')-d_{\widehat{Y}}(y,y') | \geq d'_{GH}(\widehat{X},\widehat{Y}),
\end{split}
\end{align}
where the second equality follows because the set 
$\{ |  d_X(x,x')-d_Y(y,y') |\} _{ (x,y), (x', y')\in C }$ coincides with 
$\{ |  d_{\widehat{X}}(x,x')-d_{\widehat{Y}}(y,y') | \}_{(x,y), (x', y')\in C'}$ and the last inequality follows from Definition \ref{d'gh}.

If $C'$ is not a bijection, then it is a single-valued surjective mapping from $\widehat{X}$ to~$Y$, which is not yet injective, so we perform the next step. Consider all the vertices $y\in Y$ such that $D_{C^{\prime}}(y)>1$ and double each of them $D_{C^{\prime}}(y)-1$ times. We get a new pseudo-metric space $\widehat{Y}$ and a new correspondence $\widehat{C}$ between $\widehat{X}$ and $\widehat{Y}$ which matches each double of $y$ with a single point of~$\widehat X$. Now, $\widehat{C}$ is an injective mapping from $\widehat{X}$ to $\widehat{Y}$, therefore it is a bijection. See Figure \ref{fig_bijection_C_hat} for an example of this procedure. Now the same argument as in \eqref{eq_d_GH_geq_d'_GH} establishes the claim. 
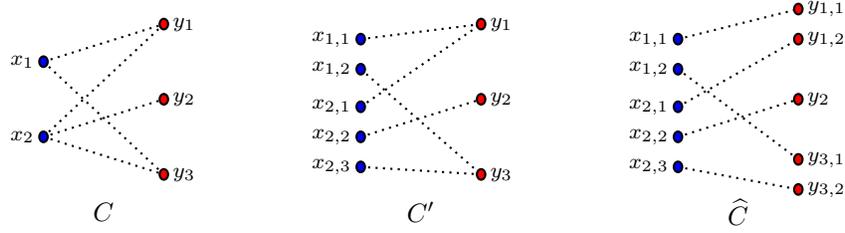
\begin{figure}
\begin{tikzpicture}[xscale=0.8]
\node[left] at (0,-0.5) {\footnotesize$x_2$};
\node[left] at (0,0.5) {\footnotesize$x_1$};
\coordinate (x2) at (0,-0.5); 
\coordinate (x1) at (0,0.5); 

\node[right] at (2,1) {\footnotesize$y_1$};
\node[right] at (2,0) {\footnotesize$y_2$};
\node[right] at (2,-1) {\footnotesize$y_3$};
\coordinate (y1) at (2,1); 
\coordinate (y2) at (2,0); 
\coordinate (y3) at (2,-1); 

\draw[dotted, thick] (x1)--(y3);
\draw[dotted, thick] (x1)--(y1);

\draw[dotted, thick] (x2)--(y1);
\draw[dotted, thick] (x2)--(y2);
\draw[dotted, thick] (x2)--(y3);

\draw[fill=blue, thick] (0,-0.5) circle (2pt);
\draw[fill=blue, thick] (0,0.5) circle (2pt);

\draw[fill=red, thick] (2,-1) circle (2pt);
\draw[fill=red, thick] (2,0) circle (2pt);
\draw[fill=red, thick] (2,1) circle (2pt);

\node at (1,-1.5) {$C$};

\begin{scope}[xshift=150]

\draw[dotted, thick] (0,0.8)--(2,1);
\draw[dotted, thick] (0,0.4)--(2,-1);

\draw[dotted, thick] (0,-0.1)--(2,1);
\draw[dotted, thick] (0,-0.5)--(2,0);
\draw[dotted, thick] (0,-0.9)--(2,-1);

\node[left] at (0,0.8) {\footnotesize$x_{1,1}$};
\node[left] at (0,0.4) {\footnotesize$x_{1,2}$};

\node[left] at (0,-0.1) {\footnotesize$x_{2,1}$};
\node[left] at (0,-0.5) {\footnotesize$x_{2,2}$};
\node[left] at (0,-0.9) {\footnotesize$x_{2,3}$};

\draw[fill=blue, thick] (0,0.8) circle (2pt);
\draw[fill=blue, thick] (0,0.4) circle (2pt);

\draw[fill=blue, thick] (0,-0.9) circle (2pt);
\draw[fill=blue, thick] (0,-0.5) circle (2pt);
\draw[fill=blue, thick] (0,-0.1) circle (2pt);

\node[right] at (2,1) {\footnotesize$y_1$};
\node[right] at (2,0) {\footnotesize$y_2$};
\node[right] at (2,-1) {\footnotesize$y_3$};

\draw[fill=red, thick] (2,-1) circle (2pt);
\draw[fill=red, thick] (2,0) circle (2pt);
\draw[fill=red, thick] (2,1) circle (2pt);

\node at (1,-1.5) {$C'$};
\end{scope}

\begin{scope}[xshift=300]
\draw[dotted, thick] (0,0.8)--(2,1.2);
\draw[dotted, thick] (0,0.4)--(2,-0.8);

\draw[dotted, thick] (0,-0.1)--(2,0.8);
\draw[dotted, thick] (0,-0.5)--(2,0);

\draw[dotted, thick] (0,-0.9)--(2,-1.2);

\node[left] at (0,0.8) {\footnotesize$x_{1,1}$};
\node[left] at (0,0.4) {\footnotesize$x_{1,2}$};

\node[left] at (0,-0.1) {\footnotesize$x_{2,1}$};
\node[left] at (0,-0.5) {\footnotesize$x_{2,2}$};
\node[left] at (0,-0.9) {\footnotesize$x_{2,3}$};

\draw[fill=blue, thick] (0,0.8) circle (2pt);
\draw[fill=blue, thick] (0,0.4) circle (2pt);

\draw[fill=blue, thick] (0,-0.9) circle (2pt);
\draw[fill=blue, thick] (0,-0.5) circle (2pt);
\draw[fill=blue, thick] (0,-0.1) circle (2pt);

\node[right] at (2,1.2) {\footnotesize$y_{1,1}$};
\node[right] at (2,0.8) {\footnotesize$y_{1,2}$};

\node[right] at (2,0) {\footnotesize$y_2$};

\node[right] at (2,-0.8) {\footnotesize$y_{3,1}$};
\node[right] at (2,-1.2) {\footnotesize$y_{3,2}$};

\draw[fill=red, thick] (2,1.2) circle (2pt);
\draw[fill=red, thick] (2,0.8) circle (2pt);

\draw[fill=red, thick] (2,0) circle (2pt);

\draw[fill=red, thick] (2,-.8) circle (2pt);
\draw[fill=red, thick] (2,-1.2) circle (2pt);

\node at (1,-1.5) {$\widehat C$};
\end{scope}
\end{tikzpicture}
\caption{From a correspondence to a bijection via iterated doublings.}
\label{fig_bijection_C_hat}
\end{figure}
\end{proof}

The proof of the last proposition also gives the following proposition. This could be helpful in proving other stability results. 

\begin{proposition}
Any correspondence between finite sets can be written as a composition of doublings and a single bijection. 
\end{proposition}

We note that the doubling of a simplicial complex does not preserve the homology of the corresponding moment-angle complexes, whereas it does preserve the double homology. See the example below and the subsequent proposition. 

\begin{example}\label{ex_double_two_points}
Let $K$ be a simplicial complex consisting of two discrete points. The corresponding moment-angle complex $\zk$ is 
\[
  D^2\times S^1 \cup_{S^1\times S^1} S^1\times D^2 \cong S^3. 
\]
Let $K'$ be the doubling of $K$ at a point of $K$. Then $K'$ is a disjoint union of a $0$-simplex and a $1$-simplex. Hence, the corresponding moment-angle complex $\mathcal{Z}_{K'}$ is 
\[
  D^2\times S^1 \times S^1 \cup_{S^1\times S^1\times S^1} S^1\times D^2\times D^2, 
\]
which can be shown to be homotopy equivalent to the wedge $S^3\vee S^3\vee S^4$. We see that $H_\ast(\zk)\neq H_\ast(\mathcal{Z}_{K'})$. 
\end{example}

\begin{proposition}\label{prop_doubling_HHZK}
Let $K'$ be obtained from $K$ by doubling a vertex. Then, 
\[
\HH_\ast(\mathcal{Z}_K) \cong \HH_\ast(\mathcal{Z}_{K'}).
\]
\end{proposition}
\begin{proof}
Let $K'$ be obtained by doubling of $K$ at $i\in [m]$. We recall the chain complex~\eqref{CH} and define the chain map 
\begin{equation*}
  \varphi \colon \CH_\ast(\mathcal{Z}_{K'}) = \Bigl(\bigoplus_{I\subset [m]\sqcup \{i'\} }\widetilde{H}_\ast(K'_I),\partial'\Bigr) \to \CH_\ast(\mathcal{Z}_{K}) = \Bigl(\bigoplus_{J\subset [m]}
  \widetilde{H}_\ast(K_J),\partial'\Bigr) 
\end{equation*}
by the property that for $\alpha\in \widetilde{H}_\ast(K'_I)$, 
\[
  \varphi(\alpha)=\begin{cases} 0 & \text{if } i'\in I;\\
  \alpha &\text{if } i' \notin I.
  \end{cases}
\]

We claim that the chain map $\varphi$ defined above is a weak equivalence. Since $\varphi$ is surjective, it suffices to show that $\ker \varphi$ is acyclic. 
Observe that 
\[
  \ker \varphi \cong \bigoplus_{i'\in I \subset [m]\sqcup \{i'\}} \widetilde{H}_\ast\left(K'_I\right)=
  \bigoplus_{L\subset[m]}\widetilde H_*(K'_{L\sqcup\{i'\}}).
\]
Let 
$T_n \colonequals \{ L \subset [m] \colon |L\setminus \{i\} | \geq n \}$
and define a decreasing filtration 
\[
  \ker\varphi=F_0\supset F_1\supset\cdots\supset F_n\supset\cdots
\]
with
\[
  F_n = \bigoplus_{L\in T_n}\widetilde H_*(K'_{L\sqcup\{i'\}}),\quad
  F_n/F_{n+1}=\bigoplus_{L\subset[m],\,
  |L\setminus\{i\}|=n}\widetilde H_*(K'_{L\sqcup\{i'\}}).
\]
The associated spectral sequence has $E^0_n=(F_n/F_{n+1},\overline{\partial'})$, where the differential $\overline{\partial'}$ induced by $\partial'$, decomposes as a direct sum
\[
  \bigoplus_{i\notin L,\,|L|=n} \Bigl( \widetilde{H}_\ast(K'_{L\sqcup\{i'\}}) \to \widetilde{H}_\ast(K'_{L\sqcup \{i,i'\}})\Bigr). 
\]
Each map in the parentheses above is an isomorphism, because $K'_{L\sqcup\{i'\}}\subset K'_{L\sqcup\{i,i'\}}$ is a homotopy equivalence. Hence, the $E_1$-page of the spectral sequence is zero, so $\ker\varphi$ is acyclic. \end{proof}

\begin{proposition}\label{cor_Wasserstein_doubling}
Suppose $\widehat{X}$ is a pseudo-metric space obtained from $X$ by an arbitrary number of doubling operations performed at arbitrary points. Then the persistent double homology modules $\PHHZ(X)$ and $\PHHZ(\widehat X)$ are isomorphic.
\end{proposition}
\begin{proof}
This follows from Proposition \ref{prop_doubling_HHZK} and obvious functorial properties of the doubling construction.
\end{proof}

\subsection{Stability for $\mathcal{PHZ}$}
First, we establish a stability result for the bigraded persistence module $\mathcal{PHZ}(X)$ of a finite pseudo-metric space $(X,d_X)$, see Definition~\ref{bgph}. 

\begin{theorem}\label{thm_BB_stable}
For two finite pseudo-metric spaces  $(X,d_X)$ and $(Y, d_Y)$ of the same cardinality, we have
\[
  d_{\IL}(\mathcal{PHZ}(X),\mathcal{PHZ}(Y))\leq
  2d_{GH}^{\prime}(X,Y),
\]
where $d'_{GH}$ is the modified Gromov--Hausdorff distance, see Definition \ref{d'gh}.
\end{theorem}
\begin{proof}
Let $\theta \colon X \to Y$ be the bijection that realizes $d'_{GH}(X,Y)$. Then, for arbitrary subset $J\subset X$,  we have 
\begin{equation}\label{eq_d_prime_subset}
d'_{GH}(J, \theta(J)) \leq d'_{GH}(X,Y).
\end{equation}
Let $\mathcal{PH}(J)$ and $\mathcal{PH}{(\theta(J))}$ be the ordinary persistent homology modules~\eqref{eq_PH} corresponding to the subspaces $J$ and 
$\theta(J)$, respectively. Then, 
\begin{equation}\label{eq_W_infty_d_ptime}
  d_{\IL}\bigl(\mathcal{PH}(J), \mathcal{PH}(\theta(J))\bigl)
    \leq 2d_{GH}(J, \theta(J)) \leq 2d'_{GH}(J, \theta(J)) \leq 2d'_{GH}(X,Y),
\end{equation}
where the first inequality follows from the stability of persistent homology (Theorem~\ref{stabph}),
the second follows by Definition~\ref{d'gh} and the third is~\eqref{eq_d_prime_subset}. Hence, 
\begin{multline*}
  d_\IL(\mathcal{PHZ}(X), \mathcal{PHZ}(Y))= 
  d_\IL\Bigl( \bigoplus_{J\subset X} \mathcal{PH}(J), \bigoplus_{J\subset X}\mathcal{PH}(\theta(J))\Bigr)
  \\
  \leq \max \bigl\{ d_\IL \bigl(\mathcal{PH}(J), \mathcal{PH}(\theta(J))\bigr) \mid J \subset X \bigr\}
  \leq 2d'_{GH}(X,Y),
\end{multline*}
where the first identity follows from~\ref{hoch-pers}, the second to last inequality follows from Proposition~\ref{oplus} and the last inequality follows from~\eqref{eq_W_infty_d_ptime}. 
\end{proof}

When using field coefficients, we also get a stability result for the bigraded barcodes using Theorem~\ref{prop_isometry}:

\begin{corollary}
Suppose $(X,d_X)$ and $(Y,d_Y)$ are two finite pseudo-metric spaces of the same cardinality. Let $\BB(X)$ and $\BB(Y)$ be the bigraded barcodes corresponding to the persistence modules $\mathcal{PHZ}(X)$ and $\mathcal{PHZ}(Y)$. Then, 
\[
  W_\infty(\BB(X), \BB(Y)) \leq 2d'_{GH}(X,Y). 
\]
\end{corollary}

The inequality in Theorem~\ref{thm_BB_stable} does not hold if we replace $d'_{GH}$ by $d_{GH}$ on the right hand side. Indeed, let $Y=X\sqcup \{x'\}$ be the doubling of $X$ at a point $x\in X$, see Definition~\ref{def_doubling_metric_sp}. Then $d_{GH}(X,Y)=0$ by Proposition~\ref{prop_d_GH_X_X'=0}. On the other hand, $d_{\IL}(\mathcal{PHZ}(X),\mathcal{PHZ}(Y))\ne0$ even when $X$ consists of two points with nonzero distance. See Example \ref{ex_double_two_points}.

\subsection{Stability for $\mathcal{PHHZ}$} 
The main result of this subsection (Theorem~\ref{thm_BB_HH_stable}) shows that the interleaving  distance between persistence modules $\mathcal{PHHZ}(X)$ and $\mathcal{PHHZ}(Y)$ defined in \eqref{eq_PHHZ} for two finite pseudo-metric spaces $(X,d_X)$ and $(Y,d_Y)$ is bounded above by the usual Gromov--Hausdorff distance between $X$ and $Y$. It also tells us that the persistent double homology module $\mathcal{PHHZ}(X)$ is more stable than the persistence module $\mathcal{PHZ}(X)$ defined by the ordinary homology of moment-angle complexes. 

\begin{theorem}\label{thm_BB_HH_stable}
Let $X$ and $Y$ be finite pseudo-metric spaces. Then
\[
d_{\IL}(\mathcal{PHHZ}(X), \mathcal{PHHZ}(Y))\leq 2d_{GH}(X,Y). 
\]
\end{theorem}
\begin{proof}
Let $\widehat X$ and $\widehat Y$ be finite pseudo-metric spaces of the same cardinality obtained as iterated doublings of $X$ and $Y$ by the procedure discussed in the proof of Proposition~\ref{prop_dGH_d_prime_GH}. Then, we have 
\begin{multline*}
  d_{\IL}(\mathcal{PHHZ}(X), \mathcal{PHHZ}(Y))
  = d_{\IL}(\mathcal{PHHZ}(\widehat X), \mathcal{PHHZ}(\widehat Y))\\ \leq 
  d_{\IL}(\mathcal{PHZ}(\widehat X), \mathcal{PHZ}(\widehat Y))\leq
  2d'_{GH}(\widehat X,\widehat Y)=2d_{GH}(X,Y).
\end{multline*}
Here the first identity follows from Proposition~\ref{cor_Wasserstein_doubling}. 
The second inequality follows from decomposition~\eqref{phhcomposite} and 
Proposition~\ref{prop_d_IL_composition}. The third inequality is Theorem \ref{thm_BB_stable} and 
the last identity is  Proposition~\ref{prop_dGH_d_prime_GH}. 
\end{proof}

When using field coefficients we also obtain stability for the bigraded barcodes using Theorem~\ref{prop_isometry}:

\begin{corollary}
Let $\BDB(X)$ and $\BDB(Y)$ be the bigraded barcodes corresponding to the persistence modules $\mathcal{PHHZ}(X)$ and $\mathcal{PHHZ}(Y)$, respectively. Then, we have 
\[
W_\infty(\BDB(X), \BDB(Y)) \leq 2d_{GH}(X,Y). 
\]
\end{corollary}

\bibliographystyle{alpha}

\end{document}